\documentclass[a4paper,12pt]{article}
\usepackage[warn]{mathtext}
\usepackage{cmap}
\usepackage[T2A]{fontenc}
\usepackage[cp1251]{inputenc}
\usepackage[russian]{babel}
\usepackage{amsmath}
\usepackage{wasysym}
\usepackage{amssymb,amsthm}
\usepackage[dvips]{graphicx}

\theoremstyle{plain}
\newtheorem{theorem}{Теорема}[section]
\newtheorem{corollary}[theorem]{Следствие}
\newtheorem{lemma}[theorem]{Лемма}
\newtheorem{proposition}[theorem]{Предложение}

\newtheorem{remark}[theorem]{Замечание}

\DeclareMathOperator{\Evol}{Evol}

\theoremstyle{definition}
\newtheorem{definition}[theorem]{Определение}

\begin{document}
\title{Периодичность морфических слов}

\author{И.~Митрофанов}

\maketitle

\begin{center}
\bigskip
Московский Государственный Университет им. М.~В.~Ломоносова.
\end{center}

\begin{abstract}
We give a proof for the decidability of the HD0L ultimate periodicity problem.
\end{abstract}

\section{Введение.}
{\it Алфавит} -- это произвольное конечное множество. Его элементы будут называться {\it буквами}.

Отображение начального куска натурального ряда в алфавит $A$ называется конечным словом (или просто словом) над $A$. Множество конечных слов, включая пустое, обозначается $A^*$. Пустое слово обозначают $\varepsilon $.
{\it Бесконечное слово}, или {\it сверхслово} -- это отображение из $\mathbb N$ в алфавит $A$. Множество сверхслов обозначается $A^{\omega}$.

Для конечного слова определена {\it длина} -- количество букв в нём. Длина слова $u$ также будет обозначаться $|u|$.
Если слово $u_1$ конечно, то определена {\it конкатенация} слов $u_1$ и $u_2$ -- слово $u_1u_2$, получающееся приписыванием второго к первому справа.

Слово $v$ является {\it подсловом} слова $u$, если $u=v_1vv_2$ для некоторых слов $v_1$, $v_2$. В случае, когда $v_1$ или $v_2$ -- пустое слово, $v$ называется {\it началом} или соответственно {\it концом} слова $u$. На словах существует естественная структура частично упорядоченного множества: $u_1\sqsubseteq u_2$, если $u_1$ является подсловом $u_2$. Будем обозначать $u_1\sqsubseteq_k u_2$, если слово $u_1$ входит в $u_2$ не менее $k$ раз.

Сверхслово $W$ называется {\it рекуррентным}, если любое его подслово встречается в $W$ бесконечно много раз, иначе говоря, $v\sqsubseteq W\Rightarrow v\sqsubseteq _{\infty}W$.

Сверхслово $W$ называется {\it периодичным}, если $W=uuuuu\dots$ для некоторого непустого $u$. Само слово $u$, равно как и его длина, называются
{\it периодом} сверхслова $W$. 
Сверхслово называется {\it заключительно периодичным}, если оно представляется в виде конкатенации $W=uW'$, где $u$ -- конечное слово, а $W'$ --
периодичное сверхслово.
\begin{proposition} Если сверхслово заключительно периодично и рекуррентно, то оно периодично.
\end{proposition}

Множество слов $A^{*}$ над алфавитом $A$ можно считать свободным моноидом с операцией конкатенацией и единицей -- пустым словом. Отображение $\varphi\colon A^{*}\to B^{*}$ называется {\it морфизмом}, если оно сохраняет операцию моноида. Очевидно, морфизм достаточно задать на буквах алфавита $A$. Морфизм называется {нестирающим}, если образом никакой буквы не является пустое слово. Если $|\varphi(a_i)|=1$ для любой буквы $a_i\in A$,
то $\varphi$ -- {\it кодирование}. Если алфавиты $A$ и $B$ совпадают, то $\varphi$ называется {\it подстановкой}.

Морфизм можно продолжить на бесконечное слово по правилу: если $u$ -- начало $W$, то $\varphi(u)$ -- начало $\varphi(W)$.

Если $\varphi$ -- такая подстановка, что $\varphi(a_1)=a_1v$ для некоторого слова $v$ и для всех $k\in \mathbb N$ $\varphi^k(v)\not=\varepsilon$, то говорится, что $\varphi$ {\it продолжается над} $a_1$ и бесконечное слово 
$$
\varphi^{\infty}(a_1):=a_1v\varphi(v)\varphi^2(v)\varphi^3(v)\varphi^4(v)\dots
$$
называется {\it чисто морфическим}, или {\it бесконечной неподвижной точкой} морфизма $\varphi$.
Если задан морфизм $h\colon A^{*}\to C^{*}$, то сверхслово $h(\varphi^{\infty}(a_1))$ называется {\it морфическим}.

Сформулируем проблему заключительной периодичности для HD0L систем: {\it cуществует ли алгоритм, который по двум произвольным морфизмам $\varphi$ и $h$
и букве $a_1$ определяет, является ли морфическое слово $h(\varphi^{\infty}(a_1))$ заключительно периодичным.}

Известно, что можно считать $\varphi$ нестирающим, а $h$ -- кодированием. Это следует из результата
\begin{theorem}[см. \cite{AS}, глава $7$] Если $f: A^*\to B^*$ и $g:A^*\to A^*$ -- произвольные морфизмы и $f(g^{\infty}(a_1))$ --
бесконечное слово, то можно найти такие алфавит $A'$, букву $a'_1\in A'$, нестирающую подстановку $\varphi$, действующую на алфавите $A$
и кодирование $\tau:A'\to B$ такие, что $f(g^\infty(a_1))=\tau(\varphi^{\infty}(a'_1))$.
\end{theorem}

Для чисто морфических слов алгоритм  был независимо построен в работах \cite{Har} и \cite{Pans_2}, в работе \cite{Honk} доказана разрешимость заключительной периодичности
для $k-$автоматных последовательностей.
Ряд схожих вопросов изучается в работе \cite{PS}.
Во время работы над текстом я узнал, что F.Durand построил \cite{Dur2} алгоритм для примитивного случая, к моменту написания текста эта работа не была опубликована.  

В настоящей работе строится алгоритм для произвольного морфического слова:
\begin{theorem}\label{main}
Проблема заключительной периодичности для HD0L-систем разрешима.
\end{theorem}
Сначала делается редукция к примитивному случаю,
а алгоритм для примитивного случая (основная часть работы) строится с использованием техники схем Рози и оснасток.
Понятия схем Рози и оснасток были определены в работе \cite{MKR} для непериодичных сверхслов; по причине того, что для произвольных рекуррентных слов
формулировки многих фактов о схемах Рози немного отличаются, в настоящей статье все факты, связанные со схемами, доказываются заново.

\section{Сведение к примитивному случаю.}

\begin{definition} Подстановка $\varphi$ называется {\it примитивной}, если существует такое $k\in \mathbb N$, что для любой буквы $a_i\in A$ в образе
$\varphi^k(a_i)$ содержатся все буквы из $A$. Если $\varphi $ продолжается над $a_1$, то сверхлово $h(\varphi^{\infty})$ для произвольного $h$ называется {\it примитивным морфическим словом}.
\end{definition}
\begin{definition} {Матрица подстановки} -- это такая матрица, у которой в $i-$м столбце в $j-$й строке стоит число вхождений буквы $a_j$ в $\varphi(a_i)$.
\end{definition}
\begin{proposition} Подстановка $\varphi$ является примитивной тогда и только тогда, когда некоторая степень её матрицы состоит из положительных
чисел.
\end{proposition}
Пусть имеется оракул, который для примитивной подстановки $g$, буквы $a$ и морфизма $f$ (то есть для тройки \{g,a,f\}) говорит, является ли $f(g^{\infty}(a))$ периодичным и если является, какой у него период.

Цель данного раздела -- для произвольных нестирающего морфизма $\varphi$, кодирования $h$ и буквы $a_1$ научиться отвечать на вопрос, является ли сверхслово
$h(\varphi^{\infty}(a_1))$ периодичным, используя этого оракула. (интересно, это одушевлённое существительное?)
Идея замены морфизма на растущий и выделения сильносвязной компоненты была почерпнута в \cite{Pr}.

\begin{theorem} \label{monad}
 Пусть $\varphi$ -- подстановка, действующая на алфавите $A$, $h$ -- морфизм из $A^*$ в $B^*$, $u$ -- конечное слово из $B^*$.
Тогда существует алгоритм, проверяющий, встречается ли $u$ в слове $W=h(\varphi^{\infty}(a_1))$ и, если встречается, конечное ли число раз.
\end{theorem}
\begin{proof} Пусть $|u|=n$. Очевидно, можно считать, что все буквы алфавита $A$ встретятся в $\varphi ^{\infty}(a_1)$. Также можно считать, что все буквы алфавита встречаются в $\varphi (a_1)$ (иначе возьмём нужную степень морфизма). Если $a\in A$, то $\chi_n(a)$ -- это число вхождений слова $u$ в $h(\varphi^k(a))$, $l_k(a)$ и $r_k(a)$ -- это
два слова длины $n$, являющиеся, соответственно, началом и коцом $h(\varphi^k(a))$. Если $|h(\varphi^k(a))|<n$, то $l_k(a)=r_k(a)=h(\varphi^k(a))$.

Обозначим $\Omega  _k$ набор из $2|A|$ слов: $r_k(a_i)$ и $l_k(a_i)$ для всех $a_i\in A$. Числа $\chi_k(a_i)$ образуют вектор $\chi_k$.

Если $\varphi(a_i)=a_{i_1}a_{i_2}\dots a_{i_m}$, то
$$
h(\varphi^{k+1}(a_i))=h(\varphi^k(a_{i_1}))h(\varphi^k(a_{i_2}))\dots h(\varphi^k(a_{i_m})).
$$
Слово $l_{k+1}$ получается, если взять первые $n$ букв слова $l_k(a_{i_1})l_k(a_{i_2})\dots l_k(a_{i_k})$. Слово $r_{k+1}$ получается, если взять последние $n$ букв слова $r_k(a_{i_1})r_k(a_{i_2})\dots r_k(a_{i_k})$. Стало быть, набор $\Omega_{k+1}$ однозначно определяется по набору $\Omega_k$. Так как различных наборов не больше, чем $|B|^{2|A|n}$, то последовательность $\Omega_k$ зациклится.

Слово $h(\varphi^{k+1}(a_i))$ делится на блоки $h(\varphi^k(a_{i_k}))$, и вхождения $u$ в $h(\varphi^{k+1}(a_i))$ бывают двух типов: те, которые целиком лежат в каком-либо блоке и те, которые принадлежат хотя бы двум блокам.
Количество первых равно $\chi_k(a_{i_1})+\chi_k(a_{i_2})+\dots +\chi_k(a_{i_m})$.
Это число равно $i-$й компоненте вектора $M \chi_k$, где $M$ -- это матрица подстановки $\varphi$.

Количество вхождений, принадлежащих хотя бы двум блокам, можно вычислить по $\Omega_k$. Таким образом, $\chi_{k+1}=A\chi_k + f(\Omega_k)$.

Пусть последовательность наборов $\Omega_k$ периодична начиная с $k=k_0$, и длина периода равна $T$. Тогда получим
$$
\chi_{k_0+(k+1)T}=M(M(\dots(M\chi_{k_0+kT}+f(\Omega_{k_0+kT}))+\dots))+f(\Omega_{k_0+(k+1)T-2}))+f(\Omega_{k_0+(k+1)T-1}).
$$
Раскроем все скобки: $\chi_{k_0+(k+1)T}=M^T\chi_{k_0+kT}+C$ для некоторого вектора $C$ с неотрицательными компонентами.

Вспомним, что нас интересует поведение первой компоненты вектора $\chi_k$ при $k$, стремящемся к бесконечности. Очевидно, $\chi_k(a_i)=0$ тогда и только тогда, когда $|\chi_k|=0$, то есть $u$ не является подсоловом $W$ тогда и только тогда, когда $C=0$ и $\chi_{t_0}=0$.

Далее надо узнать, стремится ли к бесконечности $\chi_k(a_1)$. Пусть $C\not=0$. Тогда $\chi_{k_0+kT}\geq C+MC+M^2C+\dots M^kC$ (здесь сравнение покоординатное.) Но у матриц $M^k$ первые столбцы состоит из целых положительных чисел, следовательно, $\chi_{k_0+kT}(a_1)\geq k$ и $u$ встречается в $W$ бесконечно много раз.

Пусть $C=0$. Слово $u$ встречается в $W$ бесконечное число раз тогда и только тогда, когда существует такое $i\leq |A|$, что $a_i$ встречается в $\varphi^{\infty}(a_1)$ бесконечно много раз и $\chi_{k_0}(a_i)>0$.

Если $a_1$ встречается в $\varphi (a_1)$ хотя бы два раза, то все буквы встречаются в $\varphi^{\infty}(a_1)$ бесконечно много раз. Пусть $a_1$ встречается в $\varphi (a_1)$ один раз. 
Рассмотрим ориентированный граф $G_{\varphi}$, у которого $|A|$ вершин -- буквы алфавита $A$ и из вершины, соответствующей $a_i$, в $a_j$ ведёт ребро тогда и только тогда, когда $a_j\in \varphi(a_i)$. В $G_{\varphi}$ есть петля на вершине $a_1$, удалим её. Нетрудно понять, что $a_i$ встречается в $W$ бесконечное число раз тогда и только тогда, когда в полученном графе существует самопересекающийся путь из $a_1$ в $a_i$. А это свойство графа легко проверяется.

\end{proof}

\begin{corollary} \label{ifper}
Таким образом, можно определить, верно ли, что сверхслово заключительно периодично с данным пириодом $u$,
то есть верно ли, что все слова длины $|u|$, которые бесконечно много раз встречаются в $W$, являются циклическими сдвигами $u$.
\end{corollary}

Далее $\varphi$ -- нестирающий морфизм из $A^*$ в $A^*$, $h$ -- кодирование из $A$ в $B$, $W=h(\varphi^{\infty}(a_1))$.

Слово $w\in A^*$ будем называть $\varphi-${\it ограниченным}, если последовательность
$$
w,\varphi(w),\varphi^2(w),\varphi^3(w),\dots
$$
 периодична начиная с некоторого момента.
В противном случае, $|\varphi^n(w)|\rightarrow \infty$ при $n\rightarrow \infty$ и слово $w$ называется {\it $\varphi-$растущим}.
Очевидно, слово является $\varphi-$ограниченным тогда и только тогда, когда оно состоит из $\varphi-$ограниченных букв.

\begin{theorem} \label{finwords}
Существует алгоритм, который определяет, конечно ли в $\varphi^{\infty}(a)$ число различных $\varphi-$ограниченных подслов.
Если это число бесконечно, то можно проверить, является ли $W$ периодическим.
Если оно конечно, то все эти слова алгоритмически находятся.
\end{theorem}

\begin{proof}
Прежде всего отметим, что все $\varphi-$возрастающие буквы алгоритмически находятся (см., например, \cite{Pr}). Далее $\varphi-$растущие буквы будем писать как $a_1$, $a_2$ и т.д.

Построим ориентированный граф $Q$, на рёбрах которого будут записаны упорядоченные пары слов.
Вершинами этого графа будут служить $\varphi-$растущие буквы из $A$ а также всевозможные упорядоченные пары $\varphi -$растущих букв. Введём фиктивную букву $t$, также к вершинам $Q$ добавим всевозможные пары вида $a_it$, где $a_i$ -- $\varphi-$растущая буква.

Из вершины $a_i$ в $a_j$ идёт ребро, если $a_j\in a_i$. На таких рёбрах пара слов -- $\{\varepsilon ,\varepsilon \}$ ($\varepsilon $-- пустое слово).
Из вершины $a_i$ в $a_ja_k$
ведёт ребро со словами $\{\omega,\varepsilon\}$, если для некоторого
$\varphi-$ограниченного слова $\omega $ слово $a_j\omega a_k$ является подсловом $\varphi(a_i)$ (из $a_i$ в $a_ja_k$ могут вести несколько рёбер.)
Из $a_i$ и $a_it$ ведёт по ребру с парой $\{\omega,\varepsilon\}$ в $a_jt$, если $\omega$ -- $\varphi-$ограниченное и $\varphi(a_i)$ оканчивается на $a_j\omega$.

Из $a_ia_j$ ведёт ребро в $a_ka_l$ с парой $\varphi-$ограниченныx слов $\{\omega_1;\omega_2\}$, если $\varphi(a_1)$ кончается на $a_k\omega_1$, а
$\varphi(a_2)$ начинается на $\omega_2a_l$.
\begin{proposition}
Пусть $k\in\mathbb N$. Рассмотрим какой-нибудь путь длины $k$ по рёбрам графа $Q$, выходящий из $a_1$.
Последовательность пар слов на рёбрах этого пути
$$
\{u_1,v_1\},\{u_2,v_2\},\dots,\{u_k,v_k\}. 
$$
Тогда в $\varphi^k(a_1)$ есть $\varphi-$нерасширяемое слово
$$
u_k\varphi (u_{k-1})\dots\varphi ^{k-1}(u_1)\varphi^{k-1}(v_1)\varphi^{k-2}(v_2)\dots v_k.
$$
При этом, если путь оканчивается на $a_ia_j$, то есть вхождение этого слова, обрамлённое буквами $a_i$ и $a_j$.

Наоборот, любое $\varphi-$ограниченное подслово $\varphi^k(a_1)$ можнно получить, получив слово по указанному правилу слово и взяв его подслово.
\end{proposition}

Это несложно показывается индукцией по $k$.

Первый случай: в любом ориентированном цикле графа $Q$, до которого можно добраться из $a_1$, на рёбрах цикла написаны пары пустых слов. Тогда в любом пути, выходящем из вершины $a_1$, число рёбер, на которых написаны не пустые слова, не превосходит количества вершин в $Q$. Следовательно, число различных $\varphi-$ограниченных слов конечно, а по графу $Q$ можно их всех найти.

Второй случай: есть цикл, и в этом цикле не все слова пустые. Пусть, например, в цикле есть пары слов, в которых первое слово не пустое. Тогда для 
некоторой буквы $a_i$ и непустого $\varphi-$ограниченного $u$, слово $\varphi(a_i)$ оканчивается на $a_iu$.
Образ $u$ при подстановке $\varphi$ зациклится. Следовательно, для некоторого непустого слова $U$ для любого $k\in\mathbb N$ слово $h(U^k)$ является
подсловом $W$. Тогда сверхслово $W$ заключительно периодично если и только если оно заключительно периодично с периодом $h(U)$. А проверять, есть ли 
заключительная периодичность с данным периодом умеем (см. следствие \ref{ifper}). 

\end{proof}

Пусть $I_{\varphi}$ -- множество всех $\varphi-$растущих букв, $B_{\varphi}$ -- множество $\varphi-$ограниченных подслов сверхслова
$\varphi^{\infty}(a_1)$(включая пустое).
Можно считать, что $B_{\varphi}$ конечно и что мы знаем все слова в $B_{\varphi}$.
Рассмотрим (конечный) алфавит $C$, состоящий из символов $[twt']$, где $t$ и $t'$ буквы из $I_{\varphi}$, а 
$w$ -- слово из $B_{\varphi}$ и слово $twt'$ является подсловом $\varphi^{\infty}(a)$.

Определим морфизм $\psi:C^* \to C^*$ следующим образом:
$$
\psi([twt'])=[t_1wt_2][t_2wt_3]\dots[t_kw_kt_{k+1}],
$$
где $\varphi(tw)=w_0t_1w_1t_2\dots t_kw'_k$, слово $\varphi(t')$ начинается с $w''_kt_{k+1}$ и $w_k=w'_kw''_k$ (cлова $w_i$, $w'_k$ и $w''_k$ принадлежат $B_{\varphi} $).

Также определим $f:C^*\to A^*$ по правилу
$$
f([twt'])=tw.
$$

\begin{proposition} \label{mor_hor}
 Все буквы алфавита $C$ являются $\psi-$растущими.
\end{proposition}

\begin{proof}
Заметим, что в $\psi^n([twt'])$ столько же букв, сколько в слове $\varphi^n(t)$ $\varphi-$растущих букв.
Очевидно, в образе $\varphi(t)$ от произвольной буквы $t\in I_{\varphi}$ содержится хотя бы одна буква из $I_{\varphi}$. Более того, в слове $\varphi ^n(t)$ для некоторого $n$ содержатся хотя бы две буквы из $I_{\varphi}$, иначе $\varphi ^n(t)=w_nt_{i_n}v_n$
(где $w_n$ и $v_n$ принадлежат $B_{\varphi }$) и $|\varphi^n{t}|$ ограниченно.
\end{proof}

Пусть $\varphi^{\infty}(a)$ имеет вид $a_1w_1a_2\dots$, где $a_1,a_2\in I_{\varphi}$, $w_1\in B_{\varphi}$.
Тогда, несложно убедиться, что для любого $n$ слово $\varphi ^n(a)$ является началом слова $f(\psi^n([a_1w_1a_2]))$,
следовательно, слова $W$ и $(h\circ f)(\psi^{\infty}([a_1w_1a_2]))$ совпадают. Заметим, что морфизм $h'=h\circ f$ является нестирающим.

\begin{remark} Конструкция морфизма $\psi$ встречалась в работах \cite{Pr,Pans}.
\end{remark}

Рассмотрим ориентированный граф $G_{\varphi}$, вершинами которого являются буквы алфавита $C$, и из $c_i$ ведёт стрелка в $c_j$ тогда и только тогда, когда $c_j$ содержится в $\psi(c_i)$. 
Пусть $D$ -- сильносвязная компонента этого графа, до которой можно дойти по стрелочкам из $[a_1w_1a_2]$. Тогда $\psi$ можно рассматривать как морфизм из $D^*$ в $D^*$.
Существует буква $d\in D$ и натуральное $k$ такое, что $\psi^k(d)$ начинается с $d$. Обозначим $\rho=\psi^k$. 
Слово $\rho^{\infty}(d)$ является бесконечным словом, все конечные подслова которого являются подсловами $\psi^{\infty}([a_1w_1a_2])$.
Стало быть, все конечные подслова $h'(\rho^{\infty}(d))$ являются подсловами $W$. Очевидно, что морфизм $\rho$ является примитивным
в ограничении на $D^*$.

Спросим оракула о периодичности сверхслова $h'(\psi^{\infty}(d))$. Если оно непериодично, то в $W$ бесконечно много специальных справа подслов
и само $W$ не является заключительно периодичным.

Если же оракул сказал, что сверхслово $h'(\psi^{\infty}(d))$ периодично и его период -- слово $u$, то, если $W$ заключительно периодично,
его периодом является то же самое слово $u$.
Согласно следствию \ref{ifper}, мы можем проверить, верно ли, что $h'(\psi^{\infty}(d))$ заключительно периодично с периодом $u$. Таким образом, утверждение теоремы \ref{main} верно, если разшешима проблема периодичности для примитивных морфических слов.

\section{Графы и схемы Рози.}

Пусть $W$ -- бесконечное слово. {\it Граф Рози порядка $k$} этого сверхслова обозначается $G_k(W)$, если же понятно, о каком слове идёт речь, то будем писать просто $G_k$.
Вершинами этого графа являются всевозможные различные подслова сверхслова $W$, имеющие длину $k$.
Две вершины графа $u_1$ и $u_2$ соединяются направленным ребром, если в $W$ есть такое подслово $v$, что $|v|=k+1$, $v[1]v[2]\dots v[k]=u_1$ и $v[2]v[3]\dots v[k+1]=u_2$. (Запись $v[i]$ обозначает $i-$тую букву слова $v$.)

Если $w$ -- подслово $W$ длины $k+l$, ему соответствует в $G_k$ путь длины $l$. Этот путь проходит по рёбрам, соответствующим подсловам слова $w$ длины $k+1$.

Свойства $W$, которые определяются множеством его конечных подслов, можно определить, зная последовательность его графов Рози. Так, рекуррентность $W$
означает, что все графы $G_k$ являются сильносвязными, а периодичность -- то, что при некотором $k$ граф $G_k$ является циклом. Из \ref{monad}
следует, что для морфического слова и любого $k$ по морфизму можно алгоритмически найти граф $G_k$. Поведение же последовательности графов $G_k$
определить затруднительно; для некоторого подкласса чисто морфических слов это было сделано в работе \cite{Frid}.

Схемы Рози также определяют множество его конечных подслов, но при этом их поведение проще описывается.

{\it Графом со словами} будем называть сильносвязный ориентированный граф, у которого на каждом ребре написано по два слова -- {\it переднее} и {\it заднее}. Также потребуем, чтобы каждая вершина либо имеет входящую степень $1$, а исходящую больше $1$, либо входящую степень больше $1$ и исходящую степень $1$. Вершины первого типа назовём {\it раздающими}, а второго -- {\it собирающими}.

{\it Путь} в графе со словами -- это последовательность рёбер, каждое следующее из которых выходит из той вершины, в которую входит предыдущая. {\it Симметричный путь} -- это путь, первое ребро которого начинается в собирающей вершине, а последнее ребро кончается в раздающей.

Каждый путь можно записать словом над алфавитом -- множеством рёбер графа, и это слово называется {\it рёберной записью пути}. Иногда мы будем отождествлять путь и его рёберную запись. Ребро пути $s$, идущее $i$-тым по счёту, будем обозначать $s[i]$. Для двух путей так же, как и для слов, определены отношения {\it подпути} (пишем $s_1\sqsubseteq s_2$), {\it начала} и {\it конца}.
Кроме того, пишем $s_1\sqsubseteq_k s_2$, если для соответствующих слов $u_1$ и $u_2$ -- рёберных записей путей $s_1$ и $s_2$ --  выполнено $u_1\sqsubseteq_k u_2$.
Если последнее ребро пути $s_1$ идёт в ту же вершину, из которой выходит первое ребро пути $s_2$, путь, рёберная запись которого является конкатенацией рёберных записей путей $s_1$ и $s_2$, будем обозначать $s_1s_2$.

\begin{remark}
Определения пути, подпути, рёберной записи, начала и конца имеют смысл для любых ориентированных графов. 
\end{remark}

Введём понятие {\it переднего слова $F(s)$, соответствующего пути $s$ в графе со словами.} Пусть $v_1v_2\dots v_n$ -- рёберная запись пути $s$. В $v_1v_2\dots v_n$ возьмём подпоследовательность: включим в неё $v_1$, а также те и только те рёбра, которые выходят из раздающих вершин графа. Эти рёбра назовём {\it передними образующими для пути $s$}. Возьмём передние слова этих рёбер и запишем их последовательную конкатенацию, там получаем $F(s)$.

Аналогично определяется $B(s)$. В $v_1v_2\dots v_n$ возьмём рёбра, входящие в собирающие вершины и ребро $v_n$ в порядке следования -- это {\it задние образующие для пути $s$}. Тогда последовательной конкатенацией задних слов этих рёбер получается
{\it заднее слово $B(s)$ пути $s$}.

\begin{definition} \label{Def1}
Граф со словами будет являться {\it схемой Рози} для рекуррентного сверхслова $W$, если он удовлетворяет следующим свойствам, которые в дальшейшем будут называться {\it свойствами схем Рози}:

\begin{enumerate}
    \item Граф сильносвязен и состоит более чем из одного ребра.
    \item Все рёбра, исходящие из одной раздающей вершины графа, имеют передние слова с попарно разными первыми буквами.
	Все рёбра, входящие в одну собирающую вершину графа, имеют задние слова	с попарно разными последними буквами.
    \item Для любого симметричного пути, его переднее и заднее слова совпадают. То есть можно говорить просто о слове 		симметричного пути.
    \item Если есть два симметричных пути $s_1$ и $s_2$ и выполнено $F(s_1)\sqsubseteq _k F(s_2)$, то $s_1\sqsubseteq _k s_2$.
	\item Все слова, написанные на рёбрах графа, являются подсловами $W$.
	\item Для любого $u$~-- подслова $W$ существует симметричный путь, слово которого содержит $u$.  
	\item Для любого ребра $s$ существует такое слово $u_s$, принадлежащее $W$, что любой симметричный путь, слово которого содержит $u_s$, проходит по ребру $s$.
\end{enumerate}  
\end{definition}

В работе \cite{MKR} понятие схем Рози вводилось только для непериодических слов.

\begin{lemma} Пусть $S$ -- произвольный граф со словами.
\begin{enumerate}
\item Если путь $s_1$ оканчивается в раздающей вершине и в этой же вершине начинается путь $s_2$, то $F(s_1s_2)=F(s_1)F(s_2)$.

\item Если путь $s_1$ оканчивается в собирающей вершине и в этой же вершине начинается путь $s_2$, то $B(s_1s_2)=B(s_1)B(s_2)$.
\end{enumerate}
\end{lemma}

\begin{proof}
1.Множество образующих передних рёбер пути $s_1s_2$ -- это в точности образующие передние рёбра пути $s_1$ и образующие передние рёбра пути $s_2$, записанные последовательно.

2.Множество образующих задних рёбер пути $s_1s_2$ -- это образующие задние рёбра пути $s_1$ и образующие задние рёбра пути $s_2$, записанные последовательно.
\end{proof}

\begin{corollary} \label{C2_3}
Если для графа со словами $S$ выполнено свойство $3$ схем Рози, $s$ -- симметричный путь в $S$, $s_1$ -- произвольный путь в $S$ и $s_1\sqsubseteq s$, то $F(s_1)\sqsubseteq F(s)$. 
\end{corollary}

\begin{lemma}
Пусть $S$ -- граф со словами, для которого выполнено свойство $3$ схем Рози. Если $s_1$ и $s_2$ -- два симметричных пути таких, что $s_1\sqsubseteq _k s_2$, то $F(s_1)\sqsubseteq _k F(s_2)$.
\end{lemma}

\begin{proof}
Пусть $v_1v_2\dots v_n$ -- рёберная запись пути $s_2$, а $w_1w_2\dots w_m$ -- рёберная запись $s_1$.
Для каждого вхождения слова $w_1w_2\dots w_m$ в слово $v_1v_2\dots v_n$ разобьём слово $v_1v_2\dots v_n$ на две части.
Ту часть, которая находится правее соответствующего вхождения буквы $w_m$, назовём {\it путевым концом}, а остальное -- {\it путевым началом}. Так как ребро $w_m$ входит в раздающую вершину, то любое путевое начало является симметричным путём, а первое ребро любого путевого конца является в пути $s_2$ передним образующим ребром.
Если $s_b$ и $s_e$ -- путевые начало и конец соответственно, то $F(s_2)=F(s_b)F(s_e)$, при этом $F(s_b)$ оканчивается на $F(s_1)$. Для доказательства леммы достаточно показать, что передние слова любых двух путевых концов различные.
Для любого путевого конца $s_e$ множество его передних образующих рёбер -- это те передние образующие рёбра пути $s_2$, которые находятся в $s_e$.
Но для различных путевых концов такие подмножества вложены одно в другое, а так как для любого путевого конца его первое ребро является образующим, то они вложены строго.
\end{proof}

\begin{definition}
Симмметричный путь $s$ называется {\it допустимым}, если $F(s)\sqsubseteq W$.
\end{definition}

\begin{lemma} \label{Lm2_4}
Пусть $S$ -- граф со словами, для которого в применении к сверхслову $W$ выполнены свойства $1-6$ графов Рози. Если $u\sqsubseteq W$, то в схеме $S$ есть допустимый путь $s$ такой, что $u\sqsubseteq F(s)$.
\end{lemma}

\begin{proof}
Пусть $l_{max}$ -- максимальная из длин слов на рёбрах $S$.
Так как сверхслово $W$ рекуррентно, то в нём есть вхождение слова $u$ такое, что первая буква этого вхождения имеет в $W$ номер
 более $l_{\max}$. Рассмотрим $w$ -- подслово $W$, имеющее вид $u_1uu_2$, где $|u_1|>l_{max}$, $|u_2|>l_{max}$.
Согласно свойству $6$, в схеме $S$ существует симметричный путь $s_1$ такой, что $w\sqsubseteq F(s_1)$.
Можно считать, что $s_1$ -- минимальный (относительно $\sqsubseteq $) симметричный путь с таким свойством. 

Пусть $v_1v_2\dots v_n$ -- рёберная запись этого пути. Пусть $v_{n_1}$ -- последнее ребро, являющееся передним образующим для пути $s_1$.
Это либо $v_1$, либо ребро, выходящее из раздающей вершины. Если верно первое, то в пути $s_1$ только одно переднее образующее ребро и
$|F(s_1)|\leq l_{\max}$, противоречие с длиной слова $w$.

Значит, можно рассмотреть симметричный путь с рёберной записью $v_1v_2\dots v_{n_1-1}$.
Из минимальности $s_1$ следует, что $w \not \sqsubseteq  F(v_1v_2\dots v_{n_1-1})$. 

Рассуждая аналогично, получим, что если $v_{n_2}$ -- первое ребро, являющееся задним образующим, то $w \not \sqsubseteq F(v_{n_2}v_{n_2+1}\dots v_n)$.
Докажем следующий факт: $n_2\leq n_1$. В самом деле, иначе, в силу минимальности $n_2$, среди рёбер $v_1,v_2,\dots v_{n_1}$ ни одно не входит в собирающую вершину, то есть в симметричном пути $v_1v_2\dots v_{n_1-1}$ ровно одно заднее образующее ребро (а именно $v_{n_1-1}$) и, стало быть, $|B(v_1v_2\dots v_{n_1-1})|\leq l_{\max}$. Но 
\[
F(s_1)=F(v_1v_2\dots v_{n_1-1})F(v_{n_1})=B(v_1v_2\dots v_{n_1-1})F(v_{n_1})\leq 2 l_{\max}.
\]
Противоречие с тем, что $|F(s_1)|>2 l_{\max}$.

Значит, мы можем рассматривать симметричный путь $s_2=v_{n_2}v_{n_2+1}\dots v_{n_1-1}$. Пусть $w' = F(s_2)$.
Мы можем записать $F(s_1)=B(v_{n_2-1})w'F(v_{n_1})$. Так как слово $F(s_1)$ содержит $w$, а слова $B(v_{n_2-1})w'$ и $w'F(v_{n_1})$ не содержат, то $w'\sqsubseteq w$. Стало быть, $w'\sqsubseteq W$.

С другой стороны, $F(s_1)=u'_1u_1uu_2u'_2$, где $\min\{|u'_1u_1|,|u_2u'_2|\}\geq l_{\max}$.
А так как $\max\{B(v_{n_2-1}),F(v_{n_1})\}\leq l_{\max}$, то $u\sqsubseteq w'$. Следовательно, путь $s_2$ -- искомый.

\end{proof}

\begin{lemma} \label{Lmtochn}
Пусть имеется допустимый путь $l$ со словом $u$. Если $uu_1\sqsubseteq W$ для некотогого $u_1$, то в схеме Рози $S$ есть такой допустимый путь $l'$, что его началом является $l$, а его слово начинается с $uu_1$. 
\end{lemma}
\begin{proof}
Согласно лемме  \ref{Lm2_4} в схеме $S$ существует допустимый путь $l_2$ такой, что $uu_1\sqsubseteq F(l_2)$.
Рассмотрим слово $F(l_2)$. В нём может быть несколько вхождений слова $u$, причём среди них есть такое (допустим, $k-$тое, если считать с конца), после которого сразу идёт $u_1$.

Тогда $l\sqsubseteq_k l_2$. 
Пусть $l_2=s_1ls_2$ (здесь рассматривается $k-$тое с конца вхождение пути $l$).
Рассмотрим $k$-тое с конца вхождение пути $l$ в путь $l_2$.
Первое ребро пути $l$ выходит из собирающей  вершины, а последнее ребро пути $s_2$ входит в раздающую вершину. Следовательно, $ls_2$ -- симметричный путь.
Этот путь является допустимым, так как $B(l_2)=B(s_1)B(ls_2)$ и $B(l_2)\sqsubseteq W$.

Так как $l$ имеет ровно $k$ вхождений в $ls_2$, то, согласно свойству $4$ определения схем Рози \ref{Def1} и лемме  \ref{Lmtochn}, $u=F(l)$ имеет ровно $k$ вхождений в $F(ls_2)$.
Кроме того, $F(ls_2)$ начинается с $F(l)=u$. Двум этим свойствам удовлетворяет ровно одно окончание слова $F(l_2)$. Значит, $F(ls_2)$ начинается со слова $uu_1$. Стало быть, путь $ls_2$ -- искомый.
\end{proof}

\begin{remark}
Аналогично доказывается следующий факт: пусть имеется допустимый путь $l$ со словом $u$. Если $u_1u\sqsubseteq W$ для некотогого $u_1$, то в схеме Рози $S$ есть такой допустимый путь $l'$, что его концом является $l$, а его слово кончается на $u_1u$.
\end{remark}

\begin{corollary}
Если $u$ -- биспециальное подслово $W$ такое, что оно содержит слово некоторого симметричного пути $l_1$ схемы $S$, то в схеме $S$ существует такой симметричный путь $l$, что $F(l)=u$.
\end{corollary}
\begin{proof} Пусть $u=u_1F(l_1)u_2$. Из биспециальности $u$ следует, что существуют буквы $a_1$ и $a_2$ такие, что $F(l_1)u_2a_1\sqsubseteq W$, $F(l_1)u_2a_2\sqsubseteq W$. Тогда существуют два пути $l_1s_1$ и $l_1s_2$ такие, что слово первого начинается с $F(l_1)u_2a_1$, а второго -- с $F(l_1)u_2a_2$.

Пусть $v_1v_2\dots v_{k-1}v_{k}\dots$ -- рёберная запись $l_1s_1$, а $v_1v_2\dots v_{k-1}v'_{k}\dots$ -- рёберная запись $l_1s_2$, различие происходит в ребре с номером $k$.
Очевидно, рёбра $v_k$ и $v'_k$ выходят из одной и той же вершины. Значит, эта вершина раздающая и путь $v_1v_2\dots v_{k-1}$ -- симметричный.
По свойству $2$ схем Рози (см. определение \ref{Def1}), $F(v_k)$ и $F(v'_k)$ начинаются с разных букв. Так как $F(l_1s_1)$ начинается с $F(v_1v_2\dots v_{k-1})F(v_k)$, а $F(l_1s_2)$ -- с $F(v_1v_2\dots v_{k-1})F(v'_k)$, то $F(v_1v_2\dots v_{k-1})=F(l_1)u_2$. При этом путь $v_1v_2\dots v_{k-1}$ начинается с $l_1$.

Аналогично рассуждая, найдём симметричный путь $w_1w_2\dots w_m$, концом которого является путь $l_1$ и словом которого является $u_1F(l_1)$.
Пусть $w_1w_2\dots w_m=l'l_1$. Тогда путь $l'v_1v_2\dots v_{k-1}$ является искомым.
\end{proof}

Важным поставщиком схем Рози являются графы Рози.

Фиксируем натуральное число $k$. Рассмотрим $G_k$ -- граф Рози порядка $k$ для сверхслова $W$. 
\begin{proposition} $G_k$ --
сильносвязный орграф. Если слово $W$ не является периодичным, то этот граф не является циклом.
\end{proposition}

\begin{proof}
Если $u_1$ и $u_2$ -- подслова $W$ длины $k$, то, в силу рекуррентности сверхслова $W$,
в $W$ найдётся подслово вида $a_{i_1}a_{i_2}\dots a_{i_n}$, где $a_{i_1}a_{i_2}\dots a_{i_k}=u_1$, $a_{i_{n-k+1}}a_{i_{n-k}}\dots a_{i_n}=u_2$.

Тогда слова вида $a_{i_l}a_{i_{l+1}}\dots a_{i_{l+k}}$ имеют длину $k+1$ и соответствуют рёбрам графа $G_k$, образующим путь, соединяющий вершины, соответствующие словам $u_1$ и $u_2$.
Кроме того, если бы $G_k$ был циклом длины $n$, то сверхслово $W$ было бы периодичным с периодом $n$.   
\end{proof}

Если слово $W$ периодично, то для всех достаточно больших $n$ граф $G_n$ является циклом, графы же с малыми номерами циклами могут и не являться.

\begin{remark} Дальнейшие построения проводятся для графа Рози, не являющегося циклом. Про периодичность самого сверхслова ничего не утверждается.
\end{remark}

Граф $S$ получается из $G_k$ следующими операциями:
\begin{enumerate}
\item Все простые цепи, соединяющие вершины, заменяются на единичные рёбра.
\item Каждая биспециальная вершина $a$ заменяется на ребро $v_a$. При этом все рёбра, которые раньше вели в $a$, будут вести в начало $v_a$,
а рёбра, которые раньше шли из $a$, будут выходить из конца ребра $v_a$.
\end{enumerate}

В графе $S$ все вершины являются либо раздающими, либо собирающими.


Рассмотрим те пути в графе $G_k$, которые соединяют две специальные вершины. Пусть путь ведёт из вершины $a_{i_1}a_{i_2}\dots a_{i_k}$ в вершину $a_{i_{n-k+1}}a_{i_{n-k}}\dots a_{i_n}$. Сопоставим ему слово по следующему правилу:

\begin{enumerate}
\item Если вершина, соответствующая $a_{i_1}a_{i_2}\dots a_{i_k}$, явлется раздающей,
то переднее слово пути -- это $a_{i_{k+1}}a_{i_{k+2}}\dots a_{i_n}$.
\item Если вершина, соответствующая $a_{i_1}a_{i_2}\dots a_{i_k}$, явлется в собирающей,
то переднее слово пути -- это $a_{i_1}a_{i_2}\dots a_{i_n}$.
\item Если вершина, соответствующая $a_{i_{n-k+1}}a_{i_{n-k}}\dots a_{i_n}$, является собирающей,
то заднее слово пути -- это $a_{i_1}a_{i_2}\dots a_{i_n-k}$.
\item Если вершина, соответствующая $a_{i_{n-k+1}}a_{i_{n-k}}\dots a_{i_n}$, является раздающей,
то заднее слово пути -- это $a_{i_1}a_{i_2}\dots a_{i_n}$.
\end{enumerate}

\begin{definition}
Если $S$ -- сильносвязный граф, не являющийся циклом, и $s$ -- путь в графе, то {\it естественное продолжение} пути $s$ {\it вправо} -- это минимальный путь, началом которого является $s$ и который оканчивается в раздающей вершине. 
{\it Естественное продолжение} пути $s$ {\it влево} -- это минимальный путь, концом которого является $s$ и который начинается в собирающей вершине.
\end{definition}

Очевидно, для сильносвязных нецикличных графов естественное продолжение существует всегда и единственно.
Теперь мы готовы написать слова на рёбрах $S$: если $v$ -- ребро в $S$, то в графе $G_k$ ему соответствует либо биспециальная вершина, либо простая цепь. Если это биспециальная вершина, то есть слово длины $k$, то в качестве $F(v)$ и $B(v)$ возьмём это слово. 
Если ребру $v$ соответствует простая цепь, то в качестве $F(v)$ возьмём то переднее слово, которое соответствует естественному продолжению вправо этой цепи. Заднее же слово ребра $v$ -- это заднее слово пути в $G_k$, соответствующего естественному продолжению влево рассматриваемой цепи.

В $G_k$ биспециальные вершины можно считать путями длины $0$.

\begin{proposition} \label{Prop_1}
В полученном графе со словами $S$ для любого пути $s$ слово $F(s)$ -- это переднее слово того пути графа $G_k$, который соответствует естественному продолжению вправо пути $s$. Аналогично, $B(s)$ -- это заднее слово того пути в $G_k$, который соответствует естественному продолжннию пути $s$ влево (в графе $S$).
\end{proposition}
\begin{proof}
Проведём доказательство для передних слов. В графе $S$ для любого пути $s$ его естественное продолжение вправо разбивается на естественные продолжения вправо его передних образующих рёбер, причём все естественные продолжения, начиная со второго, выходят из раздающих вершин и, стало быть, слова на них пишутся по правилу для первого из четырёх типов. Значит, слова для этих рёбер в объединении и дадут слово для длинного пути в $G_k$.
Для задних слов всё аналогично.
\end{proof}

\begin{lemma}      \label{Lm1}
Пусть $S$ -- определённый выше граф со словами. Тогда он является схемой Рози для сверхслова $W$.
\end{lemma}

\begin{proof}

{\bf Свойство 1.}
Так как $G_k$ является сильносвязным и не является циклом, то же самое верно и для $S$.

{\bf Свойство 2.}
Два пути, выходящие из одной раздающей вершины графа $S$ с различными первыми рёбрами, соответствуют путям в графе $G_k$, которые выходят из одной раздающей вершины $a_{i_1}a_{i_2}\dots a_{i_k}$, причём первые рёбра у путей разные.
Пусть вторые вершины путей -- это $a_{i_2}a_{i_3}\dots a_{i_k}b$ и $a_{i_2}a_{i_3}\dots a_{i_k}c$ соответственно.
Тогда слова этих путей начинаются на буквы $b$ и $c$ соответственно, и на эти буквы начинаются слова соответствующих путей в $S$. 

{\bf Cвойство 3} следует из того, что для симметричного пути его естественным продолжением вправо и естественным продолжением влево является он сам.
Согласно \ref{Prop_1}, передним и задним словом такого пути будут соответственно переднее и заднее слово пути
$$
a_{i_1}a_{i_2}\dots a_{k},\: a_{i_2}a_{i_3}\dots a_{k+1},\: \ldots ,
a_{i_l}a_{i_{l+1}}\dots a_{i_{l+n-1}},\: \ldots, \:a_{i_{n-k+1}}a_{i_{n-k}}\dots a_{i_n}
$$ 
графа $G_k$. Переднее слово этого пути определяется по типу $2$, а заднее -- по типу $4$, следовательно они совпадают.

Докажем {\bf свойство 4}. Пусть в $S$ нашлись два симметричных пути $s_1$ и $s_2$ такие, что $F(s_1)\sqsubseteq _l F(s_2)$.
Соответственные пути в $G_k$ обозначим $s'_1$ и $s'_2$. Очевидно, достаточно показать, что $s'_1\sqsubseteq _ls'_2$.
пути $s'_1$ и $s'_2$ выходят из собирающих вершин графа $G_k$ и входят в раздающие вершины. Передние слова этих путей -- это $F(s_1)$ и $F(s_2)$. Последовательности их $(k+1)$-буквенных подслов -- это последовательности рёбер $s'_1$ и $s'_2$. Значит, рёберная запись пути $s'_1$ встречается в рёберной записи $s'_2$ хотя бы $l$ раз. Если же $F(s_1)$--биспециальное слово длины $k$, то, очевидно, путь $s'_2$ хотя бы $l$ раз проходит по вершине
$F(s_1)$.

{\bf Свойство 5} достаточно доказать для передних слов. Пусть $v$ -- ребро схемы $S$.
Рассмотрим в $S$ естественное продолжение ребра $v$ вправо. Оно соответствует пути $s$ в $G_k$. Пусть $s$ содержит $L$ рёбер.
Первая вершина этого пути -- слово $a_{i_1}a_{i_2}\dots a_{i_k}$.
Первое ребро пути $s$ соответствует подслову $u$ вида $a_{i_1}a_{i_2}\dots a_{i_k}b$, при этом $a_{i_1}a_{i_2}\dots a_{i_k}b\sqsubseteq W$.
Очевидно, существует $u_1$ -- подслово сверхслова $W$ длины $L+k$, начинающееся с $a_{i_1}a_{i_2}\dots a_{i_k}b$. Ему соответствует путь $s'$ длины $L$ по рёбрам $G_k$ с первым ребром таким же, как у пути $s$.

Так как среди промежуточных вершин пути $s$ нет раздающих вершин (иначе $s$ не соответствует естественному продолжению ребра $v$ вправо), то $s'$ должен совпадать с путём $s$. Стало быль, слово пути $s$ является подсловом $u_1$ и, следовательно, подсловом $W$.

{\bf Свойство 6}. Существует такое число $M$,
что  любом пути длины $M$ по рёбрам $G_k$ будет хотя бы одна раздающая и хотя бы одна собирающая вершина. В самом деле, в качестве $M$ подойдёт количество вершин в $G_k$, иначе существовал бы цикл, в котором не было бы либо раздающих, либо собирающих вершин и $G_k$ был бы либо не сильносвязным, либо циклом.

Пусть $w\sqsubseteq W$. Тогда существует $w'\sqsubseteq W$ вида $w_1ww_2$, где $|w_1|=2M$, $|w_2|=2M$. Рассмотрим в $G_k$ путь, соответствующий слову $w'$.
Среди его последних $M$ вершин есть раздающая, а среди первых $M$ -- собирающая. Значит, можно взять путь, который короче не более, чем на $2M$ и соответствует симметричному пути в $S$. Этот путь соответствует подслову $w'$, которое имеет длину не менее $|w'|-2M$ и, следовательно, содержит $w$.
А это значит, что слово некоторого симметричного пути в $S$ содержит $w$.
 
Докажем {\bf свойство 7}. Пусть $v$ -- ребро в графе $S$. В графе Рози $G_k$ ребру $v$ соответствует путь $v'$, проходящий через вершины
$$
a_{i_1}a_{i_2}\dots a_{i_k},\: a_{i_2}a_{i_3}\dots a_{i_{k+1}},\: \ldots ,
a_{i_l}a_{i_{l+1}}\dots a_{i_{l+n-1}},\: \ldots, \:a_{i_{n-k+1}}a_{i_{n-k}}\dots a_{i_n}.
$$ 
Ни одна из вершин пути кроме первой и последней не является раздающей или собирающей.

Для $S$ уже доказаны свойства $1$--$6$. Для $S$ справедлива лемма \ref{Lm2_4}.
Для слова $a_{i_1}a_{i_2}\dots a_{i_k}a_{i_{k+1}}$, являющегося подсловом сверхслова $W$, найдём в $S$ допустимый путь $w$ такой, что $a_{i_1}a_{i_2}\dots a_{i_k}a_{i_{k+1}}\sqsubseteq F(w)$. Соответственный путь в графе $G_k$ содержит первое ребро пути $v'$, а следовательно, и весь $v'$. Стало быть, $w$ содержит $v$, а так как для $S$ выполнено свойство $4$, то любой симметричный путь, содержащий $F(w)$, содержит ребро $v$. 
\end{proof}

\section{Элементарная эволюция схем Рози.}
В этом разделе мы распространим определённую в \cite{MKR} эволюцию на схемы для периодических слов. Грубо говоря, сначала схемы для периодических слов выглядят так же, как и для периодических, но на некотором шаге эволюция останавливается.

Далее $W$ -- рекуррентное сверхслово, $S$ -- схема Рози для этого сверхслова.
Из сильносвязности $S$ следует, что существует ребро $v$, идущее из собирающей вершины в раздающую.
Такие рёбра будем называть {\it опорными}.

Цель этого раздела -- опредедить {\it эволюцию} $(W,S,v)$ схемы Рози $S$ по опорному ребру $v$. Опорное ребро является само по себе симметричным путём, стало быть, переднее и заднее слова этого ребра совпадают.

Пусть $\{x_i\}$ -- множество рёбер. входящих в начало $v$, а $\{y_i\}$ -- множество рёбер, идущих из конца $v$. Эти два множества могут пересекаться.
Обозначим $F(y_i)=Y_i$, $B(x_i)=X_i$, $F(v)=V$.

Рассмотрим все слова вида $X_i V Y_j$. Если такое слово не входит
в $W$, то пару $(x_i,y_j)$ назовём {\it плохой}, в противном случае -- {\it хорошей}. Также {\it хорошей} или {\it плохой} будем называть
 соответствующую тройку рёбер $(x_i,v,y_j)$.

Построим граф $S'$. Он получается из $S$ заменой ребра $v$ на $K_{\#\{x_i\},\#\{y_j\}}$, где $K_{m,n}$ -- полный двудольный граф. Более подробно: ребро $v$ удаляется, его начало заменяется на множество $\{A_i\}$ из $\#\{x_i\}$ вершин так, что для любого $i$ ребро $x_i$ идёт в $A_i$;
конец ребра $v$ заменяется на множество $\{B_j\}$ из $\#\{y_j\}$ вершин так, что для любого $j$ ребро $y_j$ выходит из $B_j$; вводятся рёбра $\{v_{i,j}\}$, соединяющие вершины множества $\{A_i\}$
с вершинами множества $\{B_j\}$.

По сравнению с $S$, у графа $S'$ нет ребра $v$, но есть новые рёбра $\{v_{ij}\}$. Остальные рёбра граф $S$ взаимно однозначно соответствуют рёбрам графа $S'$. Соответственные рёбра в первом и втором графе зачастую будут обозначаться одними и теми же буквами.

На рёбрах $S'$ расставим слова следующим образом. На всех рёбрах $S'$, кроме рёбер из $\{y_i\}$ и $\{v_{i,j}\}$, передние слова пишутся те же, что и передние слова соответственных рёбер в $S$. Для каждого $i$ и $j$, переднее слово ребра $y_j$ в $S'$ -- это $VY_j$. Переднее слово ребра $v_{i,j}$ -- это $Y_j$.
Аналогично, на всех рёбрах, кроме рёбер из $\{x_i\}$ и $\{v_{i,j}\}$, задние слова переносятся с соответствующих рёбер $S$; для всех $i$ и $j$ в качестве заднего слова ребра $x_i$ возьмём $X_iV$, а в качестве заднего слова ребра $v_{i,j}$ -- $X_i$.

Теперь построим граф $S''$. Рёбра $v_{i,j}$, соответствующие плохим парам $(x_i,y_j)$, назовём {\it плохими}, 
а все остальные рёбра графа $S'$ -- {\it хорошими}. Граф $S''$ получается, грубо говоря, удалением плохих рёбер из $S'$.
Более точно, в графе $S''$ раздающие вершины -- подмножество раздающих вершин $S'$, а собирающие вершины -- подмножество собирающих вершин $S'$.
В графе $S'$ эти подмножества -- вершины, из которых выходит более одного хорошего ребра, и вершины, в которые входит более одного хорошего ребра соответственно.
Назовём в $S'$ вершины этих двух подмножеств $S'$ неисчезающими. Рёбра в графе $S''$ соответствуют таким путям в $S'$, которые идут лишь по хорошим рёбрам, начинаются в неисчезающих вершинах, заканчиваются в неисчезающих, а все промежуточные вершины которых не являются неисчезающими.

Согласно лемме \ref{Lm4_5}, которая будет доказана ниже, $S''$ -- сильносвязный граф. Если он является циклом, то, как видно из этой же леммы, сверхслово $W$ периодично, при этом период сверхслова определяется по схеме $S$ и множеству плохих пар рёбер. В таких случаях будем говорить, что {\it элементарная эволюция (W,S,v) выявила периодичность сверхслова}.

 Если же $S''$ не является циклом, то у каждого ребра в $S''$ есть естественное продолжение вперёд и назад. Естественное продолжение вперёд соответствует пути в $S'$; переднее слово этого пути в $S'$ возьмём в качестве переднего слова для соответствующего ребра $S''$. Аналогично для задних слов: у ребра в $S''$ есть естественное продолжение влево, этому пути в $S''$ соответствует путь в $S'$. Заднее слово этого пути и будет задним словом ребра в $S''$.

\begin{definition} Построенный таким образом граф со словами $S''$ назовём {\it элементарной эволюцией} $(W,S,v)$. 
\end{definition}

\begin{theorem} \label{T1}
Пусть $W$ -- схема Рози. Тогда элементарная эволюция $(W,S,v)$ является схемой Рози для сверхслова $W$ (то есть удовлетворяет свойствам $1$---$7$ определения \ref{Def1}).
\end{theorem}

Доказательству этого и будет посвящена оставшаяся часть раздела. 

Опишем соответствие $f$ между симметричными путями (за исключением опорного ребра $v$) в схеме $S$ и симметричными путями в $S'$.
Пусть $s$ -- симметричный путь в $S$ -- имеет рёберную запись $v_1v_2\dots v_n$. Некоторые (может быть, ни одно) из этих рёбер являются ребром $v$.
Очевидно, что если ребро $v$ встречается в рёберной записи, то оно находится между рёбрами $x_i$ и $y_j$ для некоторых $i$ и $j$.
Если $v$ является первым или последним ребром пути $s$, то соответствующее $x_i$ или $y_j$ находится только с одной стороны.
Соответствие $f$ преобразует $v_1v_2\dots v_n$ (рёберную запись $s$) следующим образом:
для каждого вхождения $v$ определяется, находится ли оно на краю рёберной записи или в середине, если на конце, то ``$v$'' просто отбрасывается, а если в середине, то ``$v$'' заменяется на ``$v_{ij}$''.
Индексы $i$ и $j$ берутся те же, что и у рёбер $x_i$ и $y_j$, стоящих рядом с $v$ в рёберной записи $s$. Путь с соответствующей рёберной записью является образом $f(s)$ при соответствии.

Обратное соответствие $f^{-1}$ таково: если в графе $S'$ есть путь $l$ и его рёберная запись -- $v_1v_2\dots v_n$,
то каждое вхождение ребра ``$v_{ij}$'' заменяется на ``$v$''.
Кроме того, если $v_1$ (первое ребро пути) -- это одно из $\{y_i\}$, то в начало рёберной записи приписывается ``$v$'',
если же $v_n$ (последнее ребро пути) -- одно из рёбер $\{x_i\}$, то ``$v$'' приписывается в конец рёберной записи.

Очевидно, $f$ и $f^{-1}$ сохраняют свойство последовательности рёбер ``начало каждого следующего ребра является концом предыдущего'', то есть пути переводят в пути. Кроме того, $f$ и $f^{-1}$ на самом деле являются взаимно обратными. Далее, в графе $S'$ рёбра вида $y_j$ выходят из собирающих вершин, рёбра вида $x_i$ входят в раздающие, рёбра $v_{ij}$ выходят из раздающих вершин и входят в собирающие; в графе $S$ ребро $v$, как опорное, выходит из собирающей и входит в раздающую, рёбра $x_i$ входят в собирающую вершину, а $y_j$ выходят из  раздающей, для всех остальных рёбер графа $S$ свойство выходить из собирающей вершины или входить в раздающую сохраняется в графе $S'$. Таким образом, при соответствиях $f$ и $f^{-1}$ сохраняется симметричность путей.

\begin{proposition} \label{Pr4_2}
При этом соответствии сохраняются слова симметричных путей.
\end{proposition}
\begin{proof}
Докажем этот факт для передних слов, так как для задних всё аналогично. Автоматически будет доказано, что передние и задние слова симметричных путей в $S'$ совпадают.

Пусть в графе $S$ имеется симметричный путь $s$ с рёберной записью $v_1v_2\dots v_n$.
Разберём четыре случая.
\begin{enumerate}
\item Пусть первое ребро $v_1=v$, а последнее $v_n \not = v$. Рассмотрим все вхождения ребра $v$:
$$
s = vy_{j_1}\dots vy_{j_2}\dots vy_{j_3}\dots v_n.
$$
Пусть $f(s)=s'$.
$$
s'=y_{j_1}\dots v_{i_2j_2}y_{j_2}\dots v_{i_3j_3}y_{j_3}\dots v_n.
$$
Рассмотрим передние образующие рёбра в $s$ и $s''$. В пути $s$ в парах вида $vy_j$ окажется выбранным ребро $y_j$,
кроме первой пары, где выбраны будут оба ребра. В $s'$ в парах $v_{ij}y_j$ окажутся выбраны рёбра $v_{ij}$,
при этом первое ребро $y_{j_1}$ тоже будет передним образующим. Остальные передние образующие рёбра в обоих путях одни и те же.
Таким образом, 
$$
F(s)=F(v)F(y_{j_1})\dots F(y_{j_2})\dots F(y_{j_3})\dots = VY_{j_1}\dots Y_{j_2}\dots Y_{j_3}\dots
$$
$$
F(s')=F(y_{j_1})\dots F(v_{i_2j_2})\dots F(v_{i_3j_3})\dots = VY_{j_1}\dots Y_{j_2}\dots Y_{j_3}\dots
$$

Фрагменты, отмеченные многоточиями, одинаковы в обоих наборах.
\item Пусть $v_1=v$, $v_n = v$. Аналогично первому случаю, рассмотрим вхождения $v$ в рёберную запись пути $s$.
$$
s = vy_{j_1}\dots vy_{j_2}\dots vy_{j_3}\dots x_{i_k}v.
$$
$$
s'=f(s)=y_{j_1}\dots v_{i_2j_2}y_{j_2}\dots v_{i_3j_3}y_{j_3}\dots x_{i_k}.
$$
Тогда
$$
F(s)=F(v)F(y_{j_1})\dots F(y_{j_2})\dots F(y_{j_3})\dots = VY_{j_1}\dots Y_{j_2}\dots Y_{j_3}\dots
$$
$$
F(s')=F(y_{j_1})\dots F(v_{i_2j_2})\dots F(v_{i_3j_3})\dots = VY_{j_1}\dots Y_{j_2}\dots Y_{j_3}\dots
$$
Как мы видим, слова совпадают.
\item Пусть $v_1 \not = v$, $v_n  = v$. Аналогично, рассмотрим вхождения $v$ в рёберную запись пути $s$.
$$
s = v_1\dots vy_{j_1}\dots vy_{j_2}\dots x_{i_k}v.
$$
$$
s'=f(s)=v_1\dots v_{i_1j_1}y_{j_1}\dots v_{i_2j_2}y_{j_2}\dots x_{i_k}.
$$
Тогда
$$
F(s)=F(v_1)F(y_{j_1})\dots F(y_{j_2})\dots F(y_{j_3})\dots = F(v_1)Y_{j_1}\dots Y_{j_2}\dots Y_{j_3}\dots
$$
$$
F(s')=F(v_1)F(v_{i_1j_1})\dots F(v_{i_2j_2})\dots F(v_{i_3j_3})\dots = F(v_1)Y_{j_1}\dots Y_{j_2}\dots Y_{j_3}\dots
$$
\item Пусть $v_1 \not = v$, $v_n \not = v$. Рассмотрим рёберную запись пути $s$.
$$
s = v_1\dots vy_{j_1}\dots vy_{j_2}\dots v_n.
$$
$$
s'=f(s)=v_1\dots v_{i_1j_1}y_{j_1}\dots v_{i_2j_2}y_{j_2}\dots v_n.
$$
Тогда
$$
F(s)=F(v_1)F(y_{j_1})\dots F(y_{j_2})\dots F(y_{j_3})\dots = F(v_1)Y_{j_1}\dots Y_{j_2}\dots Y_{j_3}\dots
$$
$$
F(s')=F(v_1)F(v_{i_1j_1})\dots F(v_{i_2j_2})\dots F(v_{i_3j_3})\dots = F(v_1)Y_{j_1}\dots Y_{j_2}\dots Y_{j_3}\dots
$$
\end{enumerate}
\end{proof}

\begin{proposition}
Если в графе $S$ два симметричных пути $s_1\not = v$, $s_2\not = v$ и $s_2\sqsubseteq _n s_1$, то $f(s_1)\sqsubseteq _nf(s_2)$.
\end{proposition}
\begin{proof}
Напомним, что $v$ -- это выделенное опорное ребро. Рассмотрим рёберные записи путей $s_1$ и $s_2$.
Вторая рёберная запись входит в первую не менее $n$ раз.

Соответствие $f$ делает с этими записями следующее: каждое вхождение ``$v$'' либо отбрасываются, если оно находилось с краю рёберной записи,
либо меняется на некоторою ребро, зависящее только от соседей ``$v$'' в записи справа и слева.

Рассмотрим некое вхождение второй рёберной записи в первую. Все буквы-рёбра, не являющиеся  ``$v$'', при соответствии $f$ не меняются.
Если буква ``$v$'' находится не на конце рёберной записи $s_2$, то в обеих рёберных записях она меняется на одну и ту же букву.
Если ``$v$'' находится в начале или конце второго слова, при соответствии $f$ во втором слове она исчезает.

Следовательно, при отображении $s$ подпоследовательность рёберной записи пути $s_1$, являющаяся рёберной записью $s_2$, переходит в подпоследовательность рёберной записи $s'$, содержащую рёберную запись $s_2$. Кроме того, так как $s_2\not = v$, в рёберной записи $s_2$ есть первое с начала ребро, не меняющееся при отображении $s$ (то есть не являющееся $v$). Для каждого вхождения второй рёберной записи в первую отметим это ребро.
Рассмотрим образ первого слова при соответствии $f$. Отмеченные рёбра передут в отмеченные рёбра, стоящие в различным местах рёберной записи пути $f(s)$.
При этом образ каждого из отмеченных рёбер будет началом рёберной записи $f(s_2)$.
Таким образом, в рёберной записи первого слова будет не менее $n$ подслов, являющихся рёберной записью пути $f(s_2)$.
\end{proof}

Доказанное утверждение вкупе с предложением \ref{Pr4_2} моментально влечёт тот факт, что для графа $S'$ выполняется свойство $4$ схем Рози.

Далее, из построения $S'$ очевидно, что для любой раздающей вершины передние слова рёбер, выходящих из неё, начинаются с различных букв, а задние слова рёбер, входящих в одну собирающую вершину, кончаются на различнные буквы.

\begin{proposition} \label{Pr4_4}
Для графа со словами $S'$ выполнены свойства 1 -- 4 схем Рози.
\end{proposition}
\begin{proof}
Осталось доказать свойство $1$. В силу сильной связности $S$, в $S$ существует цикл, проходящий по всем рёбрам $S$,
а также по всем тройкам рёбер вида $x_ivy_j$. Если применить к нему отображение $f$, то получится цикл в $S'$, проходящий по всем рёбрам.
То, что $S'$ -- не цикл, очевидно, так как у него есть раздающие рёбра.   
\end{proof}

\begin{lemma} \label{big_path}
В графе $S$ существует допустимый путь, который можно разбить на три куска так, что первый и третий проходят по всем рёбрам и всем ``хорошим'' тройкам рёбер.
\end{lemma}
\begin{proof}
Составим $\{u_i\}$  -- набор подслов сверхслова $W$.
Для каждого ребра $v\in S$ включим в этот набор слово $u_s$ -- слово, соответствующее этому ребру по свойству $7$ для схем Рози.
Кроме того, включим в набор те слова вида $X_i V Y_j$, которые являются подсловами $W$.
В сверхслове $W$ существует подслово $U_1$, содержащее все слова набора $\{u_i\}$.
Пусть $l_{\max}$ -- максимальная из длин слов на ребрах $S$.
В $W$ найдётся подслово $U_2$ вида $U_1wU_1$, где $|w|>10l_{\max}$.

Применим к $U_2$ лемму \ref{Lm2_4}. Получим допустимый путь $s$ с рёберной записью $v_1v_2\dots v_n$.
Слово $F(s)$ имеет вид $w_1U_1wU_1w_2$.
Множество симметричных путей, являющихся началами пути $s$, можно упорядочить по включению:
$$
s_1\sqsubseteq s_2\sqsubseteq \dots\sqsubseteq s_k=s.
$$

Несложно понять, что у пути $s$ столько же передних образующих рёбер, сколько у него начал-симметричных путей.
Если $v_{i_1},v_{i_2},\dots,v_{i_k}$ -- передние образующие для пути $s$, то
$$
F(s_j)=F(v_{i_1})F(v_{i_2})\dots F(v_{i_j}).
$$
Значит, $F(s_{i+1})-F(s_i)\leq l_{\max}$ для любого $i$.

Рассмотрим минимальное $i$ такое, что $w_1U_1\sqsubseteq F(s_i)$.
Рёберная запись $p_1=s_i$ имеет вид $p_1=v_1v_2\dots v_{k_1}$. 
Очевидно, $F(p_1)=w_1U_1w'_1$ для некоторого $w'_1$. При этом $|w'_1|\leq l_{\max}$,
иначе слово $F(s_{i-1})$, являющееся началом $F(s_i)$ длиной не менее $|F(s_i)|-l_{\max}$, содержало бы $w_1U_1$.

Аналогично, существует такой симметричный путь $p_2$, что его рёберная запись имеет вид $v_{k_2}v_{k_2+1},\dots v_n$, а его слово
$F(s')=w'_2U_1w_2$, где $|w'_2|\leq l_{max}$.

Итак, у нас есть пути $p_1=v_1v_2\dots v_{k_1}$ и $p_2=v_{k_2}v_{k_2+1},\dots v_n$,
являющиеся соответственно началом и концом пути $s=v_1v_2\dots v_n$.
Докажем, что $k_1<k_2$. Предположим противное: $k_1\geq k_2$.
Введём обозначения: $A =B(v_1v_2\dots v_{k_2-1})$, $B=F(v_{k_2}\dots v_{k_1})$ (этот путь симметричен),
$C=F(v_{k_1+1}v_{k_1+2},\dots v_n)$. Тогда слова путей $v_1v_2\dots v_{k_1}$,
$v_{k_2}v_{k_2+1},\dots v_n$ и $v_1v_2\dots v_n$ -- это $AB$, $BC$ и $ABC$ соответственно. Следовательно, $F(s)<F(p_1)+F(p_2)$.
С другой стороны, 
$$
F(s)=|w_1U_1wU_1w_2|=|w_1U_1w'_1|+|w'_2U_1w_2|+(|w|-|w'_1|-|w'_2|)\geq F(p_1)+F(p_2)+8l_{\max}.
$$

Мы получили противоречие. Стало быть, верно неравенство $k_2>k_1$.

Так как $U\sqsubseteq F(p_1)$, то путь $p_1$ проходит по всем рёбрам $S$. Докажем, он проходит и по всем хорошим тройкам рёбер. В самом деле: если $(x_{i_1},y_{i_2})$ -- хорошая пара рёбер, то слово $X_{i_1}VY_{i_2}$ является подсловом $U_1$.
С другой стороны, $X_{i_1}VY_{i_2}$ является словом симметричного пути, образованного конкатенацией естественного продолжения влево ребра $x_{i_1}$, ребра $v$ и естественного продолжения вправо ребра $y_{i_2}$. Осталось воспользоваться свойством $4$ схем Рози, применённым к схеме $S$.

То же самое верно и для пути $p_2$.

Если симметричный путь проходит по тройке рёбер $x_{i_1}vy_{i_2}$ для плохой пары $(x_{i_1},y_{i_2})$,
то слово этого пути содержит $X_{i_1}VY_{i_2}$ в качестве подслова и путь не является допустимым. Следовательно, путь $s$, как допустимый, не проходит по плохим тройкам рёбер. 

\end{proof}

\begin{lemma} \label{Lm4_5}
Граф $S''$ сильносвязен. Если он является циклом, то сверхслово $W$ периодично и его период определяется по схеме $S$, опорному ребру и множеству плохих пар рёбер. 
\end{lemma}

\begin{proof}
Рассмотрим в графе $S'$ путь $f(s)$, где $s$ -- путь, построенный в лемме \ref{big_path}. Этот симметричный путь проходит только по хорошим рёбрам графа $S'$.
В нём можно выделить начало $f(p_1)$, конец $f(p_2)$ и среднюю часть, причём начало и конец проходят по всем хорошим рёбрам графа $S'$.
Пусть $d$ и $e$ -- рёбра схемы $S''$. Им соответствуют пути в $S'$.
Ребру $d$ -- путь $d_1d_2\dots d_{n_1}$, ребру $e$ -- путь $e_1e_2\dots e_{n_2}$.
Путь $f(p_1)$ проходит по ребру $d_1$. Из концов каждого из рёбер $d_1,d_2,\dots,d_{n_1-1}$ идёт ровно одно хорошее ребро. Следовательно, в пути $s$ после ребра $d_1$ обязаны идти рёбра $d_2$,$d_3$,\ldots, $d_{n_1}$. 
Аналогично, в $p_2$ есть вхождение ребра $e_{n_2}$,
а перед вхождением $e_{n_2}$ в $f(s)$ обязана идти последовательность рёбер $e_1e_2\dots e_{n_2-1}$.
Часть пути $f(s)$, начинающуюся с $d_1d_2\dots d_{n_1}$ и кончающуюся на $e_1e_2\dots e_{n_2}$, можно разбить на последовательные группы рёбер, образующие рёбра графа $S''$.

Таким образом, по рёбрам графа $S''$ можно добраться от любого ребра до любого другого. Следовательно, граф $S''$ сильносвязен.

Предположим, что $S''$ -- цикл. Тогда из каждой вершины графа $S'$ выходит ровно одно хорошее ребро.
Пусть всего хороших рёбер $n$. Рассмотрим какое-нибудь хорошее ребро, которое в схеме $S'$ выходит из раздающей вершины, и назовём его $d_1$.
Такие рёбра существуют, например, сойдёт ребро вида $v_{ij}$. В схеме $S'$ хорошие рёбра образуют цикл $d_1d_2\dots d_n$.
Возьмём $w_0$ -- произвольное подслово $W$. В $S$ cуществует симметричный путь $v_1v_2\dots v_n$ такой, что $w_0\sqsubseteq F(v_1v_2\dots v_n)\sqsubseteq W$.
Соответствующий путь в $S'$ должен проходит только по хорошим рёбрам -- следовательно, этот путь имеет вид
$d_{begin}(d_1d_2\dots d_e)^kd_{end}$, где участки $d_{begin}$ и $d_{end}$ состоят менее чем из $n$ рёбер.
Пусть переднее слово пути $d_1d_2\dots d_n$ -- это $D$. Тогда $w_0$ является подсловом слова $D_{begin}D^kD_{end}$, где длина слов $D_{begin}$ и $D_{end}$ не превосходит $n l_{max}$.
Пусть $N$ -- длина слова $D$. Тогда любое подслово сверхслова $W$ длины $N$ является циклическим сдвигом $D$ (если $D'$ -- подслово $W$ длины $N$, выберем $w_0$, содержащее $D$ в середине и имеющее длину $N+10\cdot n l_{max}$).
Таким образом, $W$ периодично с периодом $D$.
\end{proof}

\begin{proposition}
Для $S''$ выполнено свойство $2$ схем Рози.
\end{proposition}
\begin{proof}
В самом деле, рёбра, выходящие из одной раздающей вершины в графе $S''$, соответствуют путям,
выходящим из одной раздающей вершины графа $S'$ и имеющим различные первые рёбра. Стало быть, передние слова рёбер в $S''$ -- передние слова некоторых путей, выходящих из одной раздающей вершины графа $S'$ и имеющих различные первые рёбра. Передние слова этих путей начинаются с передних слов соответствующих первых рёбер, первые буквы которых, согласно предложению \ref{Pr4_4}, различны.
\end{proof}

Так как симметричные пути в $S''$ являются симметричными путями в $S'$ с теми же словами, а для графа $S'$ выполнены свойства $3$ и $4$, то эти же свойства выполнены и для графа $S''$.

\begin{proposition}
Для $S''$ выполнено свойство $5$.
\end{proposition}

\begin{proof}
Рассмотрим в графе $S'$ симметричный путь $f(s)$, где $s$ -- путь, построенный в лемме \ref{big_path}. Слово этого пути -- подслово сверхслова $W$.
Eсли ребру $d$ графа $S''$ в $S'$ соответствует путь $d_1d_2\dots d_k$, то $d_1d_2\dots d_k\sqsubseteq f(s)$.
Из следствия \ref{C2_3} имеем $F(d_1d_2\dots d_k)\sqsubseteq F(f(s))$ а также $B(d_1d_2\dots d_k)\sqsubseteq F(f(s))$. Осталось заметить, что $F(f(s))\sqsubseteq W$.
\end{proof}

\begin{proposition}
Для $S''$ выполнено свойство $7$.
\end{proposition}

\begin{proof}
Рассмотрим путь $f(s)$, где $s$ -- путь из леммы \ref{big_path}.
Докажем, что если $l$ -- симметричный путь в $s$ и $F(f(s))\sqsubseteq F(l)$, то для любого ребра $v\in S''$ выполняется $v\sqsubseteq l$.
В самом деле: путь $l$ является симметричным и в графе $S'$. По доказанному ранее свойству $4$ для схемы $S'$, путь $l$ содержит путь $f(s)$. А стало быть, в графе $S''$ он проходит по всем рёбрам.
\end{proof}

Для доказательства теоремы \ref{T1} осталось проверить свойство $6$. Пусть $u$ -- интересующее нас подслово сверхслова $W$.
Каждому ребру графа $S''$ в $S'$ соответствует путь, пусть все такие пути имеют длину не более $N$ рёбер, а слова всех рёбер имеют в $S$ длину не более $l_{\max}$.
Рассмотрим $w_0$ -- подслово в $W$ вида $u_1uu_2$, где $|u_1|=|u_2|=10\cdot N l_{max}$.
В $S$ существует допустимый путь $l$, слово которого содержит $w_0$. Рассмотрим в $S'$ путь $f(l)$. В этом пути содержится более $10\cdot N$ рёбер. Среди последних $N$ рёбер пути есть ребро, из конца которого выходит более одного хорошего ребра. Также среди первых $N$ рёбер пути есть ребро, в начало которого входят хотя бы два хороших ребра. Отрезок пути между этими двумя рёбрами обозначим $r$. Путь $r$ является симметричным, а его слово содержит слово $u$ в качестве подслова. Кроме того, $r$ является симметричным путём в схеме $S''$.

Итак, {\bf теорема \ref{T1} доказана}.

\section {Эволюция схем Рози.}
\begin{definition}
Пусть $S$ -- схема Рози для сверхслова $W$. На её рёбрах можно написать различные натуральные числа (или пары чисел). Такую схему мы назовём {\it нумерованной}. Если с рёбер пронумерованной схемы стереть слова, получится {\it облегчённая нумерованная схема}.
\end{definition}

\begin{definition}
{\it Метод эволюции} -- это функция, которая каждой облегчённой пронумерованной схеме (с нумерацией, допускающей двойные индексы) даёт этой же схеме новую нумерацию, такую, что в ней используются числа от $1$ до $n$ для некоторого $n$ по одному разу каждое.
\end{definition}

Зафиксируем какой-либо метод эволюции и далее не будем его менять.

Среди опорных рёбер пронумерованной схемы $S$ возьмём ребро $v$ с наименьшим номером, и совершим элементарную эволюцию $(W,S,v)$.
Эта эволюция либо выявит периодичность сверхслова, либо даст новую схему. Укажем естественную нумерацию новой схемы.

Напомним, что сначала строится схема $S'$, а потом -- $S''$. В схеме $S'$ все рёбра можно пронумеровать по следующему правилу: рёбра кроме $v$ сохраняют номера, а рёбра вида $v_{ij}$ нумеруют соответствующим двойным индексом.

Каждое ребро в схеме $S''$  -- это некоторый путь по рёбрам схемы $S'$, различным рёбрам из $S''$ соответствуют в $S'$ пути с попарно различными первыми рёбрами. Таким образом, рёбра схемы $S''$ можно пронумеровать номерами первых рёбер соответствующих путей $S'$. Теперь применим к облегчённой нумерованной схеме $S''$ метод эволюции. Получится новая облегчённая нумерованная схема, и в нумерованной (не облегчённой) схеме $S''$ перенумеруем рёбра соответственным образом.

\begin{definition}
Описанное выше соответствие, ставящее нумерованной схеме Рози другую нумерованную схему Рози (если не выявилась периодичность) назовём {\it детерменированной эволюцией}. Будем обозначать это соответствие $S''=\Evol(S).$
\end{definition}

\begin{definition}
Применяя детерменированную эволюцию к нумерованной схеме $S$ много раз, получаем последовательность нумерованных схем Рози.
Кроме того, на каждом шаге получаем множество пар чисел, задающие плохие пары рёбер. {\it Протокол детермеминованной эволюции} -- это последовательность таких множеств пар чисел и облегчённых нумерованных схем. Если на каком-то шаге выявилась периодичность, то протокол эволюции внезапно обрывается.
\end{definition}

Из определения элементарной эволюции следует
\begin{proposition}
Облегчённая нумерованная схема для $\Evol(S)$ однозначно определяется по облегчённой нумерованной схеме $S$ и множеству пар чисел,
задающие плохие пары рёбер.
\end{proposition}

\section {Схемы Рози слов с не более чем линейным показателем рекуррентности.}
\begin{remark} Мощность алфавита далее предполагается фиксированной.
\end{remark}

\begin{definition}
{\it Пословнная сложность}  $P(N)$ сверхслова $W$ -- количество различных подслов $W$ длины $N$.
\end{definition}

\begin{definition}
{\it Показатель рекуррентности} $P_2(N)$ для равномерно рекуррентного сверхслова $W$ -- минимальное число $P_2(N)$ такое, что в любом подслове сверхслова
$W$ длины $P_2(N)$ встретятся все подслова $W$ длины $N$. 
\end{definition}

\begin{definition} Число $C$ называется ограничителем рекуррентности сверхслова $W$, если $P_2(N)\leq CN$.
\end{definition}

\begin{remark}Мы не требуем какой-либо минимальности для ограничителя рекуррентности.
\end{remark}

\begin{lemma} \label{Lm6_7}
Существует такая вычислимая функция $C_{comp}:\mathbb N\to \mathbb N$, что если $W$ -- сверхслово с ограничителем рекуррентности $C_{rec}$, то $P(N)\leq C_{comp}(C_{rec})N$.
\end{lemma}
\begin{proof}
Рассмотрим конечное подслово длины $nC_{rec}$. С одной стороны, в нём содержатся все $n-$буквенные подслова слова $W$. С другой стороны, различных подслов длины $n$ в нём не больше, чем его длина.
\end{proof}

\begin{lemma} \label{compl}
Существует такая вычислимая функция $C_{cass}:\mathbb N\to \mathbb N$, что если для сверхслова $W$ и некоторого $C$ выполнено $P(N)\leq C N$, то $P(N+1)-P(N)\leq C_{cass}$.
\end{lemma}
Доказательство этого нетривиального утверждения приведено в работе \cite{Cassaigne}.

\begin{corollary}\cite{MKR} \label{cas}
Существует такая вычислимая функция $C_4:\mathbb N\to \mathbb N$, что если у сверхслова $W$ $C_{rec}$ -- ограничитель рекуррентности, то во всех графах Рози $G_k(W)$ количество вершин суммарной степени более $2$ не превосходит $C_4(C_{rec})$.
\end{corollary}

\begin{lemma} \label{sdvig1}
Существует такая вычислимая функция $C_5:\mathbb N\to (0;\:\frac{1}{2})$, что если для сверхслова $W$ выполнены свойства 
\begin{enumerate}
\item $C_{rec}$ является ограничителем рекуррентности,
\item Некоторое подслово $u$ встречается в $W$ два раза со сдвигом меньшим, чем $C_5(C_{rec})|u|$,
\end{enumerate}
то $W$ периодично с длиной периода, равной величине этого сдвига.
\end{lemma}
\begin{proof}
Пусть в $W$ есть два вхождения $u$ длины $k$, а их левые концы находятся на расстоянии $l<\frac{k}{2 C_{rec}}$.
$$
\lefteqn{\overbrace{
\phantom{a_1\dots a_{l}a_{l+1}a_{l+2}\dots a_{k'-1}a_{k'}}}^{u}}
a_1\dots a_{l}\underbrace{a_{l+1}a_{l+2}\dots a_{k'-1}a_{k'}a_{k'+1}\dots a_{k'+l}}_{u}
$$
По очереди рассматривая оба вхождения слова $u$, приходим к выводу, что фрагмент $a_1\dots a_{l}$ повторится подряд как минимум $[\frac{k}{l}]$ раз, что превышает $C_{rec}$. Любое подслово $W$ длины $l$ является подсловом слова $(a_1a_2\dots a_l)^{C_{rec}}$, что немедленно влечёт периодичность $W$ с периодом
$a_1a_2\dots a_l$.
\end{proof}

\begin{corollary} \label{sdvig}
Если сверхслово $W$ непериодично и $C_{rec}$ -- его ограничитель рекуррентности, то, если $u\sqsubseteq W$,
то для любых двух различных вхождений $u$ в $W$ их левые концы находятся на расстоянии не меньшем, чем $C_5(C_{rec})|u|$.
\end{corollary}

\begin{definition}
{\it Масштаб схемы} -- наименьшая из длин слов, написанных на опорных рёбрах этой схемы.
\end{definition}

Пусть $S$ -- схема Рози для сверхслова $W$. Возьмём собирающие вершины этой схемы. Из каждой собирающей вершины выходит единственное ребро, возьмём
его естественное продолжение вперёд. В схеме $S$ получится набор симметричных путей $\{s_i\}$. Рассмотрим раздающие вершины схемы $S$. Для каждой из них возьмём естественное продолжение назад ребра, входящего в вершину. Получим набор симметричных путей $\{t_i\}$.

\begin{lemma}      \label{Lm4}
Существует такая вычислимая функция $C_6:\mathbb N\to \mathbb N$, что если $W$ -- сверхслово с ограничителем рекурреентности $C_{rec}$,
 $S$ -- его схема, а $\{s_i\}$ -- построенный выше набор, то для любых двух путей $s_1$, $s_2$ из $\{s_i\}$ длины слов $F(s_1)$ и $F(s_2)$ отличаются не более, чем в $C_6(C_{rec})$ раз.
\end{lemma}
\begin{proof}
Так как $F(s_1)$ является передним словом первого ребра пути $s_1$, то $F(s_1)\sqsubseteq W$. Аналогично, $F(s_2)\sqsubseteq W$.
Ровно одно из рёбер пути $s_1$ входит в раздающую вершину. То же самое верно и для $s_2$.
Следовательно, случаи $s_1\sqsubseteq _2s_2$ или $s_2\sqsubseteq _2s_1$ невозможны. Значит, по свойству $4$ схем Рози, в $F(s_1)$ может быть не более одного вхождения $F(s_2)$ и наоборот.
 Если бы длины слов $F(s_1)$ и $F(s_2)$ различались более, чем в $2C_3$ раз, то одно из слов (например, $F(s_1)$) можно было бы разбить на два слова, каждое из которых содержало бы $F(s_2)$.
\end{proof}

\begin{remark}
В условиях предыдущей леммы, для любых путей из $\{t_i\}$ отношение длин их слов не превосходит $C_6(C_{rec})$. Доказательство полностью аналогично.
\end{remark}

\begin{corollary} \label{7_4}
Пусть $S$ -- схема Рози для $W$, а $C_{rec}$ -- ограничитель рекуррентности. Если $M$ -- масштаб схемы $S$, то для любого пути из $\{s_i\}$ или $\{t_i\}$ длина его слова не превосходит $C_6(C_{rec})M$.
\end{corollary}

\begin{lemma}      \label{Lm5}
Для любого пути из $\{s_i\}$, его слово является в $W$ специальным слева и
оканчивается на некоторое биспециальное слово длины не менее $M$, где $M$ -- масштаб схемы.
\end{lemma}
\begin{proof}
Пусть рассматриваемый путь имеет рёберную запись $v_1v_2\dots v_n$. Тогда ребро $v_n$ выходит из собирающей вершины, иначе естественное продолжение ребра $v_1$ было бы короче хотя бы на одно ребро. Ребро $v_n$ является опорным.

Докажем, что слово любого опорного ребра $v$ является биспециальным в $W$. Докажем, например, что $F(v)$ является специальным справа. Пусть из правого конца $v$ выходят $y_1$ и $y_2$. По свойству $7$ схем Рози для ребра $y_1$ есть $u_{y_1}$ -- подслово слова $W$ такое,
что любой симметричный путь, слово которого содержит $u_{y_1}$, проходит по ребру $y_1$. По лемме \ref{Lm2_4},
в схеме $S$ есть допустимый путь, слово которого содержит $u_{y_1}$. Рёберная запись этого допустимого пути обязана содержать $vy_1$.
Следовательно, в слово этого пути входит переднее слово пути $vy_1$, стало быть, $F(vy_1)\sqsubseteq W$.

Аналогично доказвается, что $F(vy_2)\sqsubseteq W$. Так как $F(vy_i)=F(v)F(y_i)$ и слова $F(y_i)$ имеют различные первые буквы по свойству $2$ схем Рози, $v$ является специальным справа. Специальность слева доказывается абсолютно так же.

Таким образом, $F(v_n)$ -- биспециальное. То, что $|F(v_n)|\geq M$, очевидно.

Теперь докажем, что слово пути $v_1v_2\dots v_n$ является специальным слева. Пусть $x_1$ и $x_2$ входят в начало ребра $v_1$. Для ребра $x_1$ есть $u_{x_1}$ -- такое подслово $W$, что любой симметричный путь, слово которого содержит $u_{x_1}$, содержит $x_1$.
В $S$ существует допустимый путь $l_{x_1}$, слово которого содержит $u_{x_1}$. Этот путь проходит по ребру $x_1$. Следующие $n$ рёбер этого пути могут быть только рёбра $v_1,v_2,\dots v_n$. Рассмотрим слово $F(l_{x_1})$: $B(x_1)F(v_1v_2\dots v_n)\sqsubseteq B(l_{x_1})$.
Значит, $B(x_1)F(s_1s_2\dots s_n)\sqsubseteq W$.

Аналогично, $B(x_2)F(v_1v_2\dots v_n)\sqsubseteq W$. Пользуясь свойством $2$ схем Рози для $B(x_1)$ и $B(x_2)$, получаем, что $F(v_1v_2\dots v_n)$ является специальным слева.
\end{proof}

Аналогично доказывается соответствующий факт для путей $\{t_i\}$.

\begin{corollary} \label{MvsP}
Если сверхслово периодично с длиной периода $N$, то масштаб любой схемы Рози для этого слова меньше, чем $N$.
\end{corollary}

\begin{proof}
Если масштаб схемы Рози равен $M\geq N$, то в сверхслове есть специальное справа подслово длины $N-1$, чего для периодичного с периодом $N$ слова не иожет быть.
\end{proof}

\begin{lemma}      \label{fin}
Существует такая вычислимая функция $C_7:\mathbb N\to \mathbb N$, что если $C_{rec}$ -- ограничитель рекуррентности сверслова $W$, а $S$ -- схема
Рози этого сверхслова, то количество вершин в $S$ не превосходит $C_7(C_{rec})$.
\end{lemma}

\begin{proof}
Очевидно, в схеме $S$ количество собирающих вершин равно числу путей в наборе $\{s_i\}$, количество раздающих вершин -- числу путей в наборе $\{t_i\}$.
Оценим число путей в $\{s_i\}$.
Пусть $M$ -- масштаб схемы $S$.
Рассмотрим для $W$ граф Рози $G_{k}$, где $k = [\frac M2]$. Рассмотрим слово $s_1\in \{s_i\}$.
Из \ref{7_4}, его слово $F(s_1)$ имеет длину, не превосходящую $C_6(C_{rec})M$.
Из леммы \ref{Lm5} следует, что последние $k$ букв этого слова являются словом, специальным справа, а первые $k$ букв -- словом, специальным слева.
При этом $F(s_1)$ является подсловом $W$. Все подслова в $F(s_1)$ длины $k+1$ образуют путь в $G_k$, ведущий из собирающей вершины в раздающую.
Пусть в этом пути обнаружился цикл длины $l$.

Это значит, что в $W$ есть подслово длины $k+l$, первые $k$ букв которого образуют то же слово, что и последние $k$ букв.
Применяя лемму \ref{sdvig1}, получим, что либо $l\geq C_5(C_{rec})k$, либо слово $W$ периодично с периодом $l$. Но, так как $l<k<M$, второй случай противоречит утверждению леммы \ref{MvsP}.

Слову $F(s_1)$ соответствует путь по рёбрам графа $G_{k}$ из одной специальной вершины в другую,
этот путь не может попадать в одну и ту же вершину $G_k$ менее, чем через $C_5(C_{rec})k$ шагов, а его длина не превосходит $C_6(C_{rec})M$.
Значит, одну и ту же вершину графа $G_k$ этот путь посетит не более $\frac{C_6(C_{rec})M}{C_5(C_{rec})k}$ раз.

Из леммы \ref{cas}, в графе $G_k$ может быть не более $C_4(C_{rec})$ специальных вершин. Тогда посещений специальных вершин у пути будет не более
$\frac{C_6(C_{rec})MC_4(C_{rec})}{C_5(C_{rec})k}$.
Если $|B|$ - мощность нашего алфавита, то всего различных путей, выходящих из специальной вершины, входящих в специальную вершину и содержащих не более $\frac{C_6(C_{rec})M}{C_5(C_{rec})k}$ (с учётом количества вхождений) специальных вершин, будет не более, чем 
$$
\sum_{i=1}^{\frac{C_6(C_{rec})MC_4(C_{rec})}{C_5(C_{rec})k}}C_4(C_{rec}) B^i.
$$

Для различных путей из $\{s_i\}$ их слова также различны, иначе, по свойству $4$ схем Рози, пути бы совпадали. Значит, этим путям мы поставили в соответствие различные пути в $G_k$. Поэтому верна оценка
$$
\#\{s_i\}\leq \sum_{i=1}^{\frac{C_6(C_{rec})MC_4(C_{rec})}{C_5(C_{rec})k}}C_4(C_{rec}) B^i.
$$

Аналогичная оценка выполняется для числа путей в $\{t_i\}$.
\end{proof}

\begin{corollary} \label{sl7_7}
Существует такая вычислимая функция $C_{sch}:\mathbb N\to \mathbb N$, что если взять всевозможные сверхслова с ограничителем рекуррентности $C_{rec}$
и рассмотреть все облегчённые нумерованные схемы, возникающие при детерменированных эволюциях этих слов, то количество различных схем будет конечно и не будет превосходить $C_{sch}(C_{rec})$.
\end{corollary}

\begin{proof} Количество таких схем не превосходит числа различных ориентированных графов, у которых не более, чем $C_7(C_{rec})$ вершинам, степени вершин не превосходят мощности алфавита (который у нас фиксирован), а рёбра пронумерованы начальным отрезком натуральных чисел.
\end{proof}

\begin{lemma} \label{Lm7_8}
Существует такая вычислимая функция $C_8:\mathbb N\to \mathbb N$, что если $C_{rec}$ -- ограничитель рекуррентности для сверхслова $W$, а $S_1$ и $S_2$
 -- нумерованные схемы Рози для и $S_2=\Evol(S_1)$,
то масштабы схем $S_1$ и $S_2$ относятся не более, чем в $C_9(C_{rec})$ раз.
\end{lemma}
\begin{proof}
Если $v$ -- опорное ребро схемы $S_2$, то ему соответствует некий симметричный путь $s$ по рёбрам схемы $S_1$, причём $F(s)=F(v_1)$.
Путь $s$ задаётся в $S_1$ номерами его рёбер в соответствующей облегчённой схеме.

Путь $s$ (то есть номера его рёбер в облегчённой пронумерованной схеме для $S_1$) однозначно определяются по следующему набору: \{Облегчённая нумерованная схема $S_1$, множество пар плохих рёбер в $S_1$(задаваемых номерами этих рёбер), облегчённая нумерованная схема $S_2$, номер ребра $v$ в схеме $S_2$\}.

Из \ref{sl7_7} следует, что таких наборов для слова $W$ конечное число. Более того, все такие наборы для заданного $C_{rec}$ можно явно указать и для каждого набора можно найти длину пути $s$.
Следовательно, можно найти такое число $C_8(C_{rec})$, что для для всех $S_1$, $S_2$ и $v$ длина соответствующего пути $s$ не превосходит $C_8(C_{rec})$.

А значит, $F(s)\leq C_8CM$, где $M$ -- масштаб схемы $S_1$, иначе для некоторого опорного ребра $v_0$ графа $S_1$ будет
$F(v_0)\sqsubseteq _{C_8(C_{rec})+1}F(s)$ и, по свойству $4$ схем Рози, $v_0\sqsubseteq _{C_8(C_{rec})+1}s$.

Неравенство в другую сторону ($|F(v)|=|F(s)|\geq M$) очевидно.
\end{proof}

\begin{lemma} \label{LmC_7}
Существует такая вычислимая функция $C_{10}:\mathbb N\to \mathbb N$, что если $C_{rec}$ -- ограничитель рекуррентности для сверхслова $W$, то для
любой схемы $S$ этого сверхслова длины всех слов на рёбрах этой схемы не превосходят $C_{10}(C_{rec})M$, где $M$ -- масштаб схемы.
\end{lemma}
\begin{proof}
Достаточно доказать для передних слов, для задних доказательство не будет отличаться.

Для рёбер, выходящих из собирающих вершин, существование такой константы следует из леммы \ref{Lm4}.
Рассмотрим произвольное ребро $v$, выходящее из раздающей вершины. По свойствам схем Рози, в $S$ есть допустимый путь $l_v$, проходящий через $v$.

Рассмотрим $l'_v$ -- минимальный симметричный подпуть этого пути, проходящий через $v$. Его рёберная запись имеет вид
$$
v_1v_2\dots v \dots v_n.
$$
Так как $l'_v$ -- минимальный, то среди рёбер $v_1v_2\dots v$ ровно одно выходит из собирающей вершины (а именно $v_1$). Также из минимальности следует, что среди рёбер $v\dots v_n$ ровно одно входит в раздающую вершину (а именно $v_n$).
Тогда $l'_v$ проходит не более, чем по двум опорным рёбрам.
А так как его слово является подсловом $W$, то из свойства $4$ для схем его длина $|F(l'_v)|$ менее, чем $3C_{rec}M$.
Осталось заметить, что $F(v)\sqsubseteq F(l'_v)$.
\end{proof}

\begin{lemma} \label{Lmdlin}
Существует такая вычислимая функция $C_{11}:\mathbb N\to \mathbb Q_+$, что если $C_{rec}$ -- ограничитель рекуррентности для сверхслова $W$, $S$ -- схема Рози для $W$, а $p$ -- допустимый путь, то $|F(p)|\geq C_{11}(C_{rec})MN$. Здесь $M$ -- масштаб схемы, а $N$ -- количество рёбер пути.
\end{lemma}
\begin{proof}

Пусть $v_1$ -- первое ребро произвольного допустимого пути $s$. Тогда слово $F(v_1)$ является началом $F(s)$. Естественное продолжение вправо ребра $v_1$
-- симметричный путь, принадлежащий $\{s_i\}$ (это множество путей определялось перед леммой \ref{Lm4}).
Согласно лемме \ref{Lm5}, $F(v_1)$ является специальным слева, а его длина не меньше, чем $M$.
Аналогично доказывается, что $F(s)$ оканчивается на специальное справа слово длины не менее $M$.

Следовательно, слово любого допустимого пути соответствует некоторому пути по рёбрам графа $G_{[\frac{M}{2}]}$, ведущему из какой-то левоспециальной вершины в какую-то правоспециальную.

Рассмотрим путь $p$. Среди его $N-1$ промежуточных вершин найдутся либо половина раздающих, либо половина собирающих.
Не умаляя общности, будем считать, что выполнено первое. Рассмотрим множество допустимых путей, являющихся началами пути $p$.
Их хотя бы $\frac{N-1}{2}$. Упорядочим эти пути так, чтобы слово каждого следующего пути являлось началом слова предыдущего.
$$
F(p)=F(p_0)\succeq F(p_1)\succeq \dots \succeq F(p_k), \:\:k\geq \frac{N-1}{2}.
$$

Словам $F(p_i)$ в графе $G_{[\frac{M}{2}]}$ соответствуют пути $l_i$, начинающиеся в одной собирающей вершине, кончающиеся в раздающих вершинах и такие,
что каждый следующий путь является подпутём-началом предыдущего.
В $G_{[\frac{M}{2}]}$ не более $C_7(C_{rec})$ специальных вершин (см. лемму \ref{cas}). Следовательно, через одну из раздающих вершин путь $l_0$
проходит не менее $\frac{N-1}{2C_7(C_{rec})}$ раз.

Из леммы \ref{sdvig1} следует, что либо между любыми двумя последовательными посещениями одной и той же вершины графа $G_{[\frac{M}{2}]}$ в пути $l_0$ не менее $[\frac{M}{2}]C_5(C_{rec})$ рёбер, либо $W$ периодично с периодом, меньшим $C_5(C_{rec})\frac{M}{2}$. Согласно утверждению \ref{MvsP}, второй случай невозможен (так как $C_5$ меньше единицы). Значит, $l_0$ проходит не менее, чем по $[\frac{M}{2}]C_5(C_{rec})\frac{N-1}{C_7(C_{rec})}$ рёбрам. Осталось сказать, что длина $l_0$
не более $|F(s_0)|$.
\end{proof}

\begin{lemma} \label{LmC_12}
Существует такая вычислимая функция $C_{12}:\mathbb N\to \mathbb N$, что если $C_{rec}$ -- ограничитель рекуррентности для сверхслова $W$,
$S$ -- схема Рози масштаба $M$, а $s$ -- путь, проходящий по $N$ рёбрам, то длина слова этого пути $F(s)$ не превосходит $C_{12}(C_{rec})NM$.
\end{lemma}
\begin{proof}
Если симметричный путь проходит по $N$ рёбрам, то длина его слова не превосходит $Nl_{max}$, где $l_{\max}$ -- максимальная из длин слов на рёбрах.
В силу леммы \ref{LmC_7}, $l_{max}\leq C_{10}(C_{rec})M$. Лемма доказана.
\end{proof}

\section{Оснастки и построение алгоритма для морфического случая.}

\begin{definition}
Пусть $A$ является конечным подсловом $W$, $S$ -- схема Рози для $W$. Множество симметричных путей в $S$, слова которых являются подсловами $A$, обозначим $S(A)$.
Если $S(A)$ непусто, в этом множестве есть максимальный элемент относительно сравнения $\sqsubseteq $, который будем называть {\it нерасширяемым путём}. Очевидно, для каждого пути из $S(A)$ есть путь, который содержит его и является нерасширяемым. Назовём этот путь {\it максимальным расширением}.
\end{definition}

\begin{remark} Если бы схема называлась не $S$, а, допустим, $S'$, то множество путей мы бы обозначали $S'(A)$.
\end{remark}

\begin{lemma} \label{lm72}
Если $s$ -- нерасширяемый путь в $S(A)$, то $|A|-2l_{\max}\leq |F(s)|\leq |A|$, где $l_{\max}$ -- максимальная из длин слов на рёбрах $S$.
\end{lemma}
\begin{proof}
Неравенство $|F(s)|\leq |A|$ очевидно, так как $F(s)\sqsubseteq A$.

Пусть $A$ имеет вид $u_1F(s)u_2$. Если $|F(s)|<|A|-2l_{\max}$, то либо $|u_1|>l_{\max}$, либо $|u_2|>l_{\max}$. Не умаляя общности, предположим второе. По лемме \ref{Lmtochn}
существует такой допустимый путь $s_1$, что началом его является путь $s$, а его слово имеет вид $F(s_1)=F(s)u_2u_3$.
Пусть $s$ имеет рёберную запись $v_1v_2\dots v_n$, а $s_2$ -- рёберную запись $v_1v_2\dots v_nv_{n+1}\dots v_{n+k}$. Рассмотрим путь, образованный $s$ и естественным продолжением вправо ребра $v_{n+1}$. Слово этого пути является подсловом $A$, следовательно, $s$ не является нерасширяемым в $S(A)$.
\end{proof}

\begin{lemma} \label {Lmekviv}
Существует такая вычислимая функция $C_{13}:\mathbb N\to \mathbb N$, что если $C_{rec}$ -- ограничитель рекуррентности для сверхслова $W$,
то для любой схемы Рози $S$, симметричного пути $l$ и слов $X$, $Y$, $Z$ таких, что выполнены следующие условия:
\begin{enumerate}
\item слово $XYZ$ является подсловом $W$;
\item $\min\{|X|,|Y|,|Z|\}>C_{13}(C_{rec})M$, где $M$ -- масштаб схемы $S$;
\item либо $W$ не имеет периода, либо его минимальный период не менее $3MC_{10}(C_{rec})$,
\end{enumerate}

  следующие условия эквивалентны:
\begin{enumerate}
\item $l\in S(XYZ)$
\item Существует такой симметричный путь $l'$, что $l\sqsubseteq l'$, $l'$ разбивается на три части $l'=xyz$, где $xy$ -- нерасширяемый путь для $S(XY)$, $y$ -- нерасширяемый путь для $S(Y)$, $yz$ -- нерасширяемый путь для $S(YZ)$.
\end{enumerate}
\end{lemma}

\begin{proof}
Выберем $C_{13}$ таким, чтобы при всех $n$ выполнялись неравенства:
$$
C_{13}(n)\geq 6C_{10}(n);
$$
$$
2C_{10}(n)M<C_5(n)\bigr(C_{13}(n)M-2C_{10}(n)M\bigl).
$$
Далее, если для функции не указан аргумент, подразумевается, что он равен $C_{rec}$.

Сначала докажем, что $1$ влечёт $2$.

Будем доказывать, что в качестве $l'$ можно взять максимальное расширение $l$ в $S(XYZ)$.
По лемме \ref{lm72}, $F(l')\geq |XYZ|-2l_{\max}$, а из леммы \ref{LmC_7} следует $l_{\max}\leq C_{10}M$.

Рассмотрим в $l'$ такое максимальное начало $l_1$, что $l_1\in S(XY)$.
Докажем, что $l_1$ -- нерасширяемый в $S(XY)$. Пусть $F(l')=F(l_1)u_1$. Рассмотрим вхождение слова $F(l')$ в $XYZ$, как подслова $u$.

Так как $l'$ -- нерасширяемый в $XYZ$, то 
$$
XYZ=d_1F(l')d_2,
$$
где $|d_1|\leq l_{\max}$, $|d_2|\leq l_{\max}$.
$$
XYZ=d_1F(l_1)u_1d_2.
$$

Есть два случая:
\begin{enumerate}
\item $|d_1F(l_1)|\leq |XY|$. В этом случае получаем $|d_1F(l_1)|\geq  |XY|-l_{\max}$, иначе у $l'$ есть начало $l'_1$ такое, что $l_1\sqsubseteq l'_1$ и
$|F(l'_1)|\leq |F(l_1)|+l_{\max}$, а стало быть, $F(l'_1)\sqsubseteq XY$.

Тогда $|F(l_1)|\geq |XY|-2l_{\max}\geq 2C_{13}M-2C_{10}M$.

Если есть два различных вхождения $F(l_1)$ в $XY$, то сдвиг между ними не превосходит $2l_{\max}\leq 2C_{10}M$.
Так как выполнено
	$$
	2C_{10}M<C_5(2C_{13}M-2C_{10}M),
	$$
то сверхслово $W$, согласно лемме \ref{sdvig1}, периодично с периодом не превосходящим $2C_{10}M$. Противоречие. Следовательно, есть ровно одно вхождение $F(l_1)$ в $XY$.

Предположим, что $l_1$ расширяемый в $XY$. Тогда для некоторого симметричного пути $l''_1$ выполняется $l_1\sqsubseteq l''_1$ и $F(l''_1)\sqsubseteq XY$.
Можно считать, что $l_1$ -- конец или начало $l''_1$.

	\begin{enumerate}
	\item $l''_1=l_bl_1$. Тогда $XY=e_1B(l_b)F(l_1)e_2$.

	Так как есть ровно одно вхождение $F(l_1)$ в $XY$, то $d_1=e_1B(l_b)$ и 
	$$
	XYZ=e_1B(l_b)F(l')d_2.
	$$
	Тогда симметричный путь $l_bl'\in S(XYZ)$ и путь $l'$ расширяемый.
	\item $l''_1=l_1l_e$. Тогда $XY=e_1F(l_1)F(l_e)e_2$. Так как у $F(l_1)$ только одно вхождение в $XY$, то $e_1=d_1$ и $$XY=d_1F(l_1)F(l_e)e_2.$$
	
	Следовательно, $F(l_1l_e)\sqsubseteq XY$.
	Кроме того, $F(l_1l_e)$ является началом $F(l')$. Тогда, по $2$-му свойству схем Рози, $l_1l_e$ является началом $l'$ (иначе рассмотрим первое ребро пути $l_1l_e$, не совпадающее с ребром пути $l'$). Противоречие.
	\end{enumerate}

\item $|d_1F(l_1)|>|XY|$. В этом случае $d_1F(l_1)=XYd_3$, где $|d_3|\leq|d_1|\leq l_{\max}$.
Так как
$F(l_1)\sqsubseteq _2XYd_3\sqsubseteq W$, $|XYd_3|\leq|XY|+C_{10}M$, $|F(l_1)|>|XY|-C_{10}M$, то слово $F(l_1)$, имеющее длину не менее $2C_{13}M-C_{10}M$, встречается со сдвигом, меньшим $2C_{10}M$. Из леммы \ref{sdvig1} выводим противоречие.
\end{enumerate}

Рассмотрим $l_2$ -- такой максимальный конец пути $l'$, что $l_2\in S(YZ)$. Аналогичными рассуждениями получим, что $l_2$ -- нерасширяемый в $S(YZ)$.

Докажем следующий факт: пути $l_1$ и $l_2$ внутри $l'$ перекрываются, то есть в путях $l_1$ и $l_2$ суммарно больше рёбер,
чем в пути $l'$. В самом деле, иначе возьмём $s$ -- естественное продолжение влево последнего ребра $l_1$.
Путь $s$ является симметричным, $F(S)\leq C_{10}M$.
Тогда $|F(l')|+|F(s)|\geq F(l_1)+F(l_2)$, то есть $$|XYZ|\geq (|XY|-2C_{10}M)+(|YZ|-2C_{10}M)-C_{10}M,$$
или $5C_{10}M\geq C_{13}M$, что неверно в силу выбора $C_{13}$.


Итак, можно считать, что $l_1=xy$, $l_2=yz$, $l'=xyz$. Для доказательства импликации $1\to 2$ осталось показать, что путь $y$ -- нерасширяемый путь в $S(Y)$.

Имеем $XYZ=d_1B(x)F(y)F(z)d_2$, $XY=d_1B(x)F(y)e_1$ и $YZ=e_2F(y)F(z)d_2$, где $\max\{|d_1|,|d_2|,|e_1|,|e_2|\}\leq l_{\max}$.
Пользуясь очевидным равенством $|XYZ|+|Y|=|XY|+|YZ|$, получаем 
$$
|Y|=|e_1|+|e_2|+|F(y)|.
$$
Так как $Y$ оканчивается на $F(y)e_1$ и начинается на $e_2F(y)$, можно сделать вывод, что $Y=e_2F(y)e_1$.

Слово $F(y)$ имеет длину не меньшую, чем $C_{13}M-2MC_{10}$, и, если бы оно встречалось в $Y$ два раза, то встречалось бы со сдвигом, не превосходящим
$2C_{10}M$. В силу леммы \ref{sdvig1} и неравенства 
$$
C_{13}M-2MC_{10}>2MC_{10},
$$
это влекло бы периодичность $W$ с периодом, меньшим $2C_{10}M$. Поэтому $F(y)$ встречается в $Y$ один раз.

Если путь $y$ -- расширяемый для $S(Y)$, то, не умаляя общности, $y$ -- начало симметричного пути $yl_e$ такого, что $F(y)F(l_e)\sqsubseteq Y$.
В таком случае, $Y=f_1F(y)F(l_e)f_2$, $e_2=f_1$ и $e_1=F(l_e)f_2$.

Тогда $XY=d_1B(x)F(y)F(l_e)f_2$, то есть $l_1l_e\sqsubseteq S(XY)$. Кроме того, $F(l_1l_e)$ является началом слова $F(l_1)e_1=B(x)F(y)e_1$, которое,
в свою очередь, является началом $B(x)F(y)F(z)=F(l')$. Следовательно, $l_1l_e$ -- начало $l'$. Противоречие.

Теперь докажем, что $2$ влечёт $1$.

Так как $xy$ -- нерасширяемый путь в $S(XY)$, а $y$ -- нерасширяемый путь в $Y$, $XY$ имеет вид $XY=d_1B(x)F(y)d_2$, а $Y$ имеет вид $Y=e_1F(y)e_2$, причём $|d_i|,|e_i|\leq l_{\max}$.
Так как $C_{10}M<C_5(C_{13}M-C_{10}M)$, то $d_2=e_2$.

Аналогично, $YZ=f_1F(y)F(z)f_2$, где $f_1=e_1$. Тогда $XYZ=d_1B(x)F(y)F(z)d_2$, а значит, $F(l)\sqsubseteq F(l')=F(xyz)\sqsubseteq XYZ$.
\end{proof}

\begin{lemma} \label{findP}
Существует такая вычислимая функция $C_{per}:\mathbb N\times\mathbb N\to \mathbb N$, что если $C_{rec}$ -- ограничитель рекуррентности для сверхслова $W$, $S$ -- схема Рози для сверхслова $W$, а $W$ периодично с длиной периода, не превосходящей $KM$, где $M$ -- масштаб схемы, то в схеме $S$ для любого допустимого пути его рёберная запись периодична с длиной периода, не превосходящей $C_{per}(C_{rec},K)$. 
\end{lemma}
\begin{remark}
Считается, что конечное слово $aabaabaabaaba$ периодично с длиной периода $3$.
\end{remark}
\begin{proof} Пусть $s$ -- допустимый путь, длина которого достаточно велика. Найдётся $s'$ -- допустимый путь, являющийся началом этого пути
и такой, что длина его слова $F(s')$ лежит в промежутке $[KM;\:KM+C_{10}(C_{rec})M]$.
Если два подслова слова $W$ начинаются на $F(s')$, то одно из них является началом другого. Следовательно, если два допустимых пути имеют общим началом $s'$, то один из них является началом другого. 
Если длина некоторого подслова $W$ не меньше $2C_{rec}(K+C_{10}(C_{rec}))M$, то это подслово содержит хотя бы два вхождения слова $F(s')$.
Возьмём такое минимальное начало $s''$ пути $s$, что $|F(s'')|\geq 2C_{rec}(K+C_{10}(C_{rec}))M$.
Из минимальности $s''$ следует, что $|F(s'')|\leq 2C_{rec}(K+C_{10}(C_{rec}))M+C_{10}(C_{rec})M$.
В таком случае, согласно лемме \ref{Lmdlin}, в пути $s''$ не более $\frac{2C_{rec}(K+C_{10}(C_{rec}))+C_{10}(C_{rec})}{C_{11}(C_{rec})}$ рёбер.

Очевидно, $s'\sqsubseteq _2s''$. Рассмотрим рёберную запись пути $s$.
В ней на расстоянии, меньшем, чем $\frac{2C_{rec}(K+C_{10}(C_{rec}))+C_{10}(C_{rec})}{C_{11}(C_{rec})}$, встречаются две рёберных записи пути $s'$. А так как среди любых двух допустимых путей, начинающихся на $s'$, один является началом другого, то рёберная запись $s$ периодична с периодом, меньшим $\frac{2C_{rec}(K+C_{10}(C_{rec}))+C_{10}(C_{rec})}{C_{11}(C_{rec})}$.
\end{proof}

\begin{remark} В дальнейшем запись $C_{per}(C_{rec})$ будет обозначать $C_{per}(C_{rec},3 C_{10}(C_{rec}))$.
\end{remark}

\begin{definition} \label{d7_4}
{\it Набор проверочных слов порядка $k$} -- это упорядоченный набор слов $\{q_i\}$, в который входят $\psi(\varphi^k(a_i))$ для всех букв алфавита $\{a_i\}$, а также
$\psi(\varphi^k(a_ia_j))$ для всевозможных пар последовательных букв слова $\varphi^{\infty}(a_1)$.
Каждая буква алфавита, а также все двубуквенные слова, являющиеся подсловами $\varphi^{\infty}(a_1)$, назовём {\it источниками}.
\end{definition}

Для любого $k$ проверочных слов порядка $k$ столько же, сколько и источников.

Пусть $W=\psi(\varphi^{\infty}(a_1))$, где подстановка $\varphi$ примитивная.
Начнём строить алгоритм, определяющий его периодичность.

Прежде всего отметим, что все источники, согласно \ref{monad}, находятся алгоритмически.

\begin{lemma}\cite{MKR} \label{Lm6_4}
По примитивному морфизму $\varphi$ алгоритмически находятся такие числа $C_1>0$, $C_2>0$, что для некоторого $\lambda_0>1$ $C_1\lambda_0^k<|\phi^k(a_j)|<C_2\lambda_0^k$ для любых $k$ и буквы $a_j$.
\end{lemma}

\begin{corollary} \label{pr7_5}
Можно найти такие положительные $C_{14}$, $C_{15}$, что в наборе проверочных слов с номером $k$ длина любого слова $q$ удовлетворяет
двойному неравенству $C_{14}\lambda_0^k<|q|<C_{15}\lambda_0^k$.
\end{corollary}

\begin{lemma}\cite{MKR,AS}
По примитивному $\varphi$ и произвольному $\psi$ алгоритмически находится число $C_{rec}$, являющееся ограничителем рекуррентности сверхслова
$\psi(\varphi^{\infty}(a_1))$.
\end{lemma}
\begin{remark} Для всех введённых ранее функций (например, $C_5$) мы будем писать $C_5$, подразумевая $C_5(C_{rec})$, где $C_{rec}$ -- вычисленный
для $W$ показатель рекуррентности.
\end{remark}


\begin{definition}
Определим {\it оснастку} $(S,k)$. Здесь $k\in \mathbb{N}$ -- порядок оснастки, $S$ -- пронумерованная схема Рози.
Для получения оснастки берётся набор проверочных слов порядка $k$ и для каждого проверочного слова $q_i$ рассматривается набор путей $S(q_i)$. Каждый путь в схеме задаётся можно задать упорядоченным набором чисел -- номерами рёбер.

Сама оснастка -- это следующий набор информации:
\begin{enumerate}
\item облегчённая нумерованная схема $S$ (то есть без слов, но с циферками на рёбрах),
\item Для каждого источника $p_i$ указывается, какие пути находятся в множестве  $S(\psi(\varphi^k(p_i)))$.
Пути указываются на облегчённой нумерованной схеме $S$ упорядоченными наборами чисел.
\end{enumerate}

\end{definition}

\begin{definition}
{\it Размер оснастки} -- максимальная длина (в рёбрах) по всем путям из $S(q_i)$ для всех $q_i$ -- проверочных слов порядка $k$.
\end{definition}

\begin{remark} Глядя на оснастку, её размер мы определить можем. А $k$ (то есть порядок проверочных слов) -- не можем.
\end{remark}

\begin{lemma} \label{size}
Можно найти такие положительные $C_{16}$, $C_{17}$ и $C_{18}$, что для любой схемы $S$ размер оснастки $(S,k)$ заключён между
$C_{16}\frac{\lambda_0^k}{M}-C_{17}$ и $C_{18}\frac{\lambda_0^k}{M}$.
\end{lemma}
\begin{proof}
Длина каждого проверочного слова, согласно \ref{pr7_5}, хотя бы $C_{14}\lambda_0^k$. Если $q_i$ -- проверочное, то длина слова нерасширяемого в
$S(q_i)$ пути не менее $C_{14}\lambda_0^k -- 2C_{10}M$. А длина этого пути в рёбрах составляет, согласно лемме \ref{lm72}, не менее $\frac{C_{14}\lambda_0^k - 2C_{10}M}{C_{12}M}$.

С другой стороны, длина слова нерасширяемого пути не может быть больше $C_{15}\lambda_0^k$, а его длина в рёбрах, как следует из леммы \ref{Lmdlin}, не может быть более $\frac{C_{15}\lambda_0^k}{C_{11}M}$.
\end{proof}

\begin{corollary} \label{wts}
Можно указать такие $C_{19}$, $C_{20}$ и $C_{21}$, что если размер оснастки $(S,k)$ больше $C_{19}$, то
\begin{enumerate}
\item Размер оснастки не меньше, чем $2C_{per}C$.
\item Длины всех проверочных слов составляют не менее $C_{13}M$.
\item Для любой хорошей тройки рёбер (то есть хорошей пары рёбер, соединённой опорным ребром) существует проходящий через неё путь, слово которого содержится во всех проверочных словах.
\item Размер оснастки $(S,k)$ относится к размеру оснастки $(\Evol(S),k)$ не более, чем в $C_{20}$ раз (если, конечно, $\Evol(S)$ существует.)
\item Размеры оснасток $(S,k)$ и $(S,k+1)$ отличаются не более, чем в $C_{20}$ раз.
\item Размер оснастки $(S,k+C_{21})$ больше размера оснастки $(S,k)$ хотя бы в два раза.
\end{enumerate}
\end{corollary}
\begin{proof}

\begin{enumerate}
\item Достаточно взять большой размер оснастки.
\item Из лемм \ref{size} и \ref{pr7_5} следует, что если $x$ -- размер оснастки, а $|q|$ -- размер проверочного слова $q$, то
$|q|>C_{14}\lambda_0^k>C_{14}(C_{18}xM)$.
\item Симметричный путь, получающийся естественным расширением хорошей тройки рёбер сначала вправо, а потом влево,
является допустимым и его слово имеет длину не более $3l_{\max}\leq 3C_{10}M$.
Если проверочное слово $q$ имеет длину хотя бы $3C_{10}MC$, где $C$ -- показатель рекуррентности, то подсловом $q$ является и слово рассматриваемого пути. Как видно из предыдущего пункта, выбором $C_{19}$ этого легко добиться.
\item Пусть $M$ и $M_1$ -- масштабы схемы $S$ и $\Evol(S)$, а $x$ и $x_1$ -- размеры оснасток. Согласно лемме \ref{Lm7_8}, $M_1\leq C_8M$.
Тогда $x_1>C_{16}\frac{\lambda_0^k}{C_8M}-C_{17}>C_{16}\frac{C_{18}x}{C_8}-C_{17}$, что при достаточно большом $x$ больше, чем $C_{16}\frac{C_{18}x}{2C_8}$.
\item Пусть $x$ и $x_1$ -- размеры оснасток $(S,k)$ и $(S,k+1)$.

Тогда $C_{16}\frac{\lambda_0^{k+1}}{M}-C_{17}<x_1<C_{18}\frac{\lambda_0^{k+1}}{M}$. С другой стороны, $\frac{x+C_{17}}{C_{16}}>\frac{\lambda_0^{k}}{M}>\frac{x}{C_{18}}$.

Таким образом, $C_{16}\frac{\lambda_0x}{C_{18}M}-C_{17}<x_1<C_{18}\frac{\lambda_0(x+C_{17})}{C_{16}M}$.

\item Пусть $x$ и $x_1$ -- размеры оснасток $(S,k)$ и $(S,k+C_{21})$.

Тогда $C_{16}\frac{\lambda_0^{k+C_{21}}}{M}-C_{17}<x_1$. С другой стороны, $\frac{\lambda_0^{k}}{M}>\frac{x}{C_{18}}$.

Таким образом, $x_1>C_{16}\frac{\lambda_0^{C_{21}}x}{C_{18}M}-C_{17}$, что не меньше $2x$ при достаточно большом $C_{21}$.

Заметим, что точное значение $\lambda_0$ мы не знаем, но из леммы \ref{Lm6_4} мы можем найти какую-то оценку снизу, отделяющую его от $1$, и какую-то оценку сверху.
\end{enumerate}
\end{proof}

\begin{lemma} \label{lm79}
Если размер оснастки $(S,k)$ не меньше, чем $C_{19}$, то, зная оснастку $(S,k)$, можно сделать хотя бы один из двух выводов:
\begin{enumerate}
\item Сверхслово $W$ является периодичным и длина периода не превосходит $MC_{per}C_{10}$, где $M$ -- масштаб схемы $S$.
(Заметим, что по оснастке оценить $M$ мы не можем, так как в оснастке используется облегчённая схема).
\item Вывод о том, какова оснастка $(S,k+1)$.
\end{enumerate}
\end{lemma}
\begin{proof}
Пусть $R$ -- размер оснастки $(S,k)$.
Среди путей оснастки есть путь, у которого $R\geq 2C_{per}C$ рёбер. Посмотрим на его рёберную запись и определим, является ли она периодичной с длиной периода, не превосходящей $C_{per}$. Если является, то в $W$ встречается $u^{C}$ для некоторого $u$, при этом длина $|u|\leq MC_{per}C_{10}$.
Можно сделать первый вывод (о переодичности), так как $C$ любое слово длины $|u|$ встречается в $u^C$.

Далее считаем, что рёберная запись этого пути не является периодичной. В таком случае из леммы \ref{findP} следует, что либо $W$ непериодично,
либо его минимальный период больше, чем $3C_{10}M$, то есть выполнено третье условие леммы \ref{Lmekviv}.

Для того, чтобы узнать оснастку $(S,k+1)$, нужно для каждого источника $p_i$ определить множество путей в
$S(\psi(\varphi^{k+1}(p_i)))$. Согласно лемме \ref{wts}, размер оснастки $(S,k+1)$ не будет больше $C_{20}R$.
Поэтому достаточно проверить принадлежность к $S(\psi(\varphi^{k+1}(p_i)))$ для каждого симметричного пути с не более, чем $C_{20}R$ рёбрами.

Пусть $\varphi(p_i)=a_{i_1}a_{i_2}\dots a_{i_k}$.  
$$
\psi(\varphi^{k+1}(p_i))=\psi(\varphi^k(a_{i_1}))\psi(\varphi^k(a_{i_2}))\dots \psi(\varphi^k(a_{i_k})).
$$ 

Каждая из букв $a_{i_1}$, $a_{i_2}$, \dots, $a_{i_k}$ является источником.
Также источниками являются $a_{i_1}a_{i_2}$, $a_{i_2}a_{i_3}$,\dots, $a_{i_{k-1}}a_{i_k}$.

Согласно лемме \ref{wts}, длины слов $\psi(\varphi^k(a_{i_1}))$, $\psi(\varphi^k(a_{i_2}))$, $\psi(\varphi^k(a_{i_3}))$ не менее $C_{13}M$,
поэтому для схемы $S$,
слов $\psi(\varphi^k(a_{i_1}))$, $\psi(\varphi^k(a_{i_2}))$, $\psi(\varphi^k(a_{i_3}))$ и произвольного симметричного пути $l$
выполняются условия леммы \ref{Lmekviv}.

Из оснастки $(S,k)$ нам известны следующие множества путей: 
$$
S(\psi(\varphi^k(a_{i_1}))\psi(\varphi^k(a_{i_2}))),\: S(\psi(\varphi^k(a_{i_2}))\psi(\varphi^k(a_{i_3}))),\: S(\psi(\varphi^k(a_{i_2}))).
$$
В каждом из этих множеств находим максимальные по включению элементы -- нерасширяемые пути.

Симметричный путь $l$ принадлежит множеству $S(\psi(\varphi^k(a_{i_1}a_{i_2}a_{i_3})))$ тогда и только тогда, когда
{\it существует такой симметричный путь $l'$, что $l\sqsubseteq l'$, а $l'$ разбивается на три части $x$,$y$ и $z$ такие, что
$xy$ -- нерасширяемый путь для $S(\psi(\varphi^k(a_{i_1}a_{i_2})))$, $yz$ -- для $S(\psi(\varphi^k(a_{i_2}a_{i_3})))$,
а $y$ -- для $S(\psi(\varphi^k(a_{i_2})))$.} А это свойство легко проверяется по оснастке $(S,k)$.

После определения множества $S(\psi(\varphi^k(a_{i_1}a_{i_2}a_{i_3})))$ воспользуемся леммой \ref{Lmekviv}, применённой к словам 
$\psi(\varphi^k(a_{i_1}a_{i_2}))$, $\psi(\varphi^k(a_{i_3}))$, $\psi(\varphi^k(a_{i_4}))$ и аналогичным способом определим множество
$S(\psi(\varphi^k(a_{i_1}a_{i_2}a_{i_3}a_{i_4})))$.

Потом определим множество $S(\psi(\varphi^k(a_{i_1}a_{i_2}a_{i_3}a_{i_4}a_{i_5})))$ и, действуя подобным образом, доберёмся до
$S(\psi(\varphi^k(a_{i_1}a_{i_2}\dots a_{i_k})))$.

\end{proof}

\begin{lemma} \label{Lmvtr}
Если размер оснастки $(S,k)$ не менее $C_{19}$, то по оснастке $(S,k)$ однозначно определяются плохие пары рёбер, а также определяется, выявит ли элементарная эволюция периодичность слова $W$. Если не выявит, то определяется оснастка $(\Evol(S),k)$.
\end{lemma}

\begin{proof}

Из леммы \ref{wts} следует, что по оснастке определяется множество хороших и плохих пар рёбер: по хорошим тройкам рёбер пути из оснастки проходят, а по плохим -- нет (так как все пути в оснастках допустимые).
Следовательно, определяется, существует ли $\Evol(S)$ и какая у $\Evol(S)$ (в случае существования) облегчённая нумерованная схема. Симметричные пути в $\Evol(S)$ соответствуют некоторым симметричным путям в $S$ с теми же словами. При этом мы можем указать соответствие, глядя лишь на облегчённые схемы.
Таким образом, для каждого симметричного пути в облегчённой нумерованной схеме $\Evol(S)$ по оснастке $(S,k)$ можно определить, каким из $\Evol(S)(A_i)$ этот путь принадлежит. (Здесь $A_i$ -- проверочные слова порядка $k$).
\end{proof}

Симметричным путям в $\Evol(S)$ соответствуют симметричные пути в $S$ с теми же словами, при этом в $\Evol(S)$ пути не длиннее, чем 
соответственные пути в $S$. Стало быть, размер оснастки $(\Evol(S),k)$ не более, чем размер оснастки $(S,k)$.

Запустим основную часть алгоритма. Возьмём $T=2 C_{19} C_{20}^{C_{21}}$.
Построим какую-нибудь схему Рози и возьмём такой порядок проверочных слов, чтобы размер оснастки был не менее $T$.
Далее будем работать только с оснастками. Первая оснастка у нас есть, а каждая следующая оснастка определяется по предыдущей.
Если размер оснастки менее $T$, то по оснастке вида $(S,k)$ строится оснастка вида $(S,k+1)$ (операция перехода описана в лемме \ref{lm79}).
Иначе делается операция перехода от оснастки вида $(S,k)$ к оснастке вида $(\Evol(S),k)$ (операция описана в лемме \ref{Lmvtr}).
Таким образом, размер оснастки никогда не упадёт ниже $C_{19}$.
С другой стороны, размер оснастки не поднимается выше $C_{20}T$.

Значит, различных оснасток в последовательности не может быть больше, чем некоторое наперёд заданное число $H$
(так как различных нумерованных облегчённых схем тоже конечно.)
Следовательно, есть альтернатива: либо не более чем за $H$ шагов последовательность оснасток зациклится,
либо она оборвётся не позднее шага с номером $H$.

Если она обрывается, значит, либо в одном из переходов согласно лемме \ref{lm79} был сделан вывод о периодичности, либо в одном из переходов
согласно лемме \ref{Lmvtr} было обнаружено, что элементарная эволюция выявляет периодичность. И в том и в другом случае, проследив переходы на необлегчённых схемах, легко обнаружить период сверхслова $W$.

Если же последовательность зацикливается, то последовательность оснасток периодична (так как каждая следующая определяется по предыдущей).
Если с некоторого момента совершаются только переходы типа $(S,k)\to(S,k+1)$, то размер оснасток неограниченно увеличивается, чего не может быть.
Следовательно, совершается бесконечно много переходов типа $(S,k)\to(\Evol(S),k)$.
То есть у сверхслова $W$ бесконечно много схем Рози. Докажем, что оно непериодично.

	В самом деле, пусть оно периодично. Тогда длины всех слов ра рёбрах всех схем одновременно ограничены. Так как облегчённых схем конечное число,
	то какая-то схема (как граф со словами) должна повториться дважды. Но у каждой следующей схемы множество слов всех симметричных путей
	согласно \ref{Pr4_2} строго содержится в аналогичном множестве для предыдущей схемы.
	
Таким образом, проблема периодичности для примитивных морфических слов разрешима.

\end{document}